\documentclass[reqno,12pt]{amsart}

\usepackage{amsmath}
\usepackage{amssymb}
\usepackage{amsfonts}
\usepackage{mathtools}
\usepackage{enumerate}
\usepackage{comment}

\allowdisplaybreaks

\usepackage{amsthm}
\usepackage{amscd}
\usepackage{graphicx}

\usepackage{palatino}
\usepackage[T1]{fontenc}
\usepackage[mathscr]{eucal}
\usepackage{dsfont}

\usepackage{fullpage}

\usepackage{acronym}
\usepackage{latexsym}
\usepackage{paralist}
\usepackage{wasysym}
\usepackage{xspace}

\usepackage[dvipsnames,svgnames]{xcolor}
\colorlet{MyBlue}{DodgerBlue!60!Black}
\colorlet{MyGreen}{DarkGreen!85!Black}

\usepackage[font=small,labelfont=bf]{caption}
\usepackage{subfigure}
\usepackage{tikz}
\usetikzlibrary{calc,patterns}

\usepackage[authoryear,comma]{natbib}

\usepackage{hyperref}
\hypersetup{
	colorlinks=true,
	linktocpage=true,
	pdfstartview=FitH,
	breaklinks=true,
	pdfpagemode=UseNone,
	pageanchor=true,
	pdfpagemode=UseOutlines,
	plainpages=false,
	bookmarksnumbered,
	bookmarksopen=false,
	bookmarksopenlevel=1,
	hypertexnames=true,
	pdfhighlight=/O,
	urlcolor=MyBlue!60!black,linkcolor=MyBlue!70!black,citecolor=DarkGreen!70!black, 
	pdftitle={},
	pdfauthor={},
	pdfsubject={},
	pdfkeywords={},
	pdfcreator={pdfLaTeX},
	pdfproducer={LaTeX with hyperref}
}

\numberwithin{equation}{section}  
\usepackage[sort&compress,capitalize,nameinlink]{cleveref}
\crefname{app}{Appendix}{Appendices}

\crefrangeformat{equation}{\upshape(#3#1#4)\textendash(#5#2#6)}

\usepackage[textwidth=30mm]{todonotes}
\usepackage{soul}
\setstcolor{red}
\sethlcolor{SkyBlue}

\newcommand{\debug}[1]{{\color{black}#1}}

\theoremstyle{plain}
\newtheorem{theorem}{Theorem}
\newtheorem{corollary}[theorem]{Corollary}
\newtheorem*{corollary*}{Corollary}
\newtheorem{lemma}[theorem]{Lemma}
\newtheorem{proposition}[theorem]{Proposition}

\theoremstyle{definition}
\newtheorem{definition}[theorem]{Definition}
\newtheorem*{definition*}{Definition}

\newtheorem*{hypothesis*}{Hypothesis}

\theoremstyle{remark}
\newtheorem{remark}[theorem]{Remark}
\newtheorem*{remark*}{Remark}
\newtheorem*{notation*}{Notational remark}

\numberwithin{theorem}{section}

\newcommand{\x}{\boldsymbol{x}}
\newcommand{\y}{\boldsymbol{y}}

\newcommand{\Expect}{\mathbf{\debug E}}
\newcommand{\Prob}{\mathbf{\debug P}}
\renewcommand{\Pr}{\mathsf{\debug P}}

\DeclareMathOperator{\Exp}{\mathsf{\debug{E}}}

\newcommand{\var}{{\rm Var} }

\newcommand{\whp}{\textit{whp}}
\newcommand{\ind}{\mathds{1}}
\newcommand{\p}{\mathfrak{p}}
\DeclareMathOperator{\tx}{\textnormal{\texttt{tx}}}

\newcommand{\cB}{\ensuremath{\mathcal B}} 
 
\newcommand{\cD}{\ensuremath{\mathcal D}} 
\newcommand{\cE}{\ensuremath{\mathcal E}} 
\newcommand{\cF}{\ensuremath{\mathcal F}} 
\newcommand{\cG}{\ensuremath{\mathcal G}} 
\newcommand{\cH}{\ensuremath{\mathcal H}} 
 
\newcommand{\cJ}{\ensuremath{\mathcal J}}

\newcommand{\cP}{\ensuremath{\mathcal P}}

\newcommand{\cT}{\ensuremath{\mathcal T}}

\newcommand{\E}{\ensuremath{\mathbb{E}}}

\newcommand{\N}{\ensuremath{\mathbb{N}}}
\newcommand{\Z}{\ensuremath{\mathbb{Z}}}

\newcommand{\R}{\ensuremath{\mathbb{R}}}
\renewcommand{\P}{\ensuremath{\mathbb{P}}}

\newcommand{\dd}{\mathrm{d}}
\newcommand{\ee}{\mathrm{e}}
 
\def\({\left(}
\def\){\right)}
\def\[{\left[}
\def\]{\right]}

\newacro{NE}{Nash equilibrium}
\newacroplural{NE}[NE]{Nash equilibria}
\newacro{PNE}{pure Nash equilibrium}
\newacroplural{PNE}[PNE]{pure Nash equilibria}
\newacro{MNE}{mixed Nash equilibrium}
\newacroplural{MNE}[MNE]{mixed Nash equilibria}
\newacro{PFNE}{prior-free Nash equilibrium}
\newacroplural{PFNE}[PFNE]{prior-free Nash equilibria}
\newacro{WE}{Wardrop equilibrium}
\newacroplural{WE}[WE]{Wardrop equilibria}
\newacro{SO}{socially optimum}
\newacro{SU}{social utility}
\newacro{BEq}{best equilibrium}
\newacro{WEq}{worst equilibrium}
\newacro{KKT}{Karush\textendash Kuhn\textendash Tucker}
\newacro{OD}[O/D]{origin-destination}
\newacro{PoA}{price of anarchy}
\newacro{PoS}{price of stability}
\newacro{PoCS}{price of correlated stability}
\newacro{BPR}{bureau of public roads}
\newacro{FIP}{finite improvement property}
\newacro{CLT}{central limit theorem}
\newacro{BPG}{buck-passing game}
\newacro{SBPG}{stochastic buck-passing game}
\newacro{MBPG}{mixed extension of the buck-passing game}

\newcounter{constant}
\newcommand{\nc}[1]{\refstepcounter{constant}\label{#1}}
\newcommand{\uc}[1]{c_{\ref{#1}}}
\begin{document}
\title[Voter model on regular random graphs]{Discordant edges for the voter model\\ on regular random graphs}

\author{Luca Avena}
\address{DIMAI, Dipartimento di Matematica e Informatica “Ulisse Dini”, Universit\'a degli studi di Firenze, Viale Giovanni Battista Morgagni 67/a, 50134 Firenze, Italy.}
\email{luca.avena@unifi.it}

\author{Rangel Baldasso}
\address{Department of Mathematics, PUC-Rio, Rua Marqu\^es de S\~ao Vicente 225, G\'avea, 22451-900 Rio de Janeiro, RJ - Brazil.}
\email{rangel@mat.puc-rio.br}

\author{Rajat Subhra Hazra}
\address{Mathematical Institute, Leiden Univerisity, P.O.\ Box 9521, 2300RA Leiden, The Netherlands}
\email{r.s.hazra@math.leidenuniv.nl}

\author{Frank den Hollander}
\address{Mathematical Institute, Leiden Univerisity, P.O.\ Box 9521, 2300RA Leiden, The Netherlands}
\email{denholla@math.leidenuniv.nl}

\author{Matteo Quattropani}
\address{Dipartimento di Matematica ``Guido Castelnuovo'', Sapienza Universit\`a di Roma, Piazzale Aldo Moro 5, 00185 Roma, Italy}
\email{matteo.quattropani@uniroma1.it}

\begin{abstract}
	We consider the two-opinion voter model on a regular random graph with $n$ vertices and degree $d \geq 3$. It is known that consensus is reached on time scale $n$ and that on this time scale the volume of the set of vertices with one opinion evolves as a Fisher-Wright diffusion. We are interested in the evolution of the number of discordant edges (i.e., edges linking vertices with different opinions), which can be thought as the perimeter of the set of vertices with one opinion, and  is the key observable capturing how consensus is reached. We show that if initially the two opinions are drawn independently from a Bernoulli distribution with parameter $u \in (0,1)$, then on time scale $1$ the fraction of discordant edges decreases and stabilises to a value that depends on $d$ and $u$, and is related to the meeting time of two random walks on an infinite tree of degree $d$ starting from two neighbouring vertices. Moreover, we show that on time scale $n$ the fraction of discordant edges moves away from the constant plateau and converges to zero in an exponential fashion. Our proofs exploit the classical dual system of coalescing random walks and use ideas from \cite{CFRmultiple} built on the so-called First Visit Time Lemma. We further introduce a novel technique to derive concentration properties from weak-dependence of coalescing random walks on moderate time scales.
  

\end{abstract}

\maketitle
	
	

\section{Model, literature and results}
\label{sec:intro}

Random processes on random graphs constitute a research area that poses many challenges. In the past decade, considerable progress has been made in understanding how the geometry of the graph affects the evolution of the process. In terms of the choice of graph, the focus has been on the Erd\"os-R\'enyi random graph, the configuration model, the preferential attachment model, and the exponential random graph, either directed or undirected, and in regimes ranging from sparse to dense. In terms of the choice of process, the focus has been on percolation, random walk, the stochastic Ising model, the contact process, and the voter model. What makes the area particularly interesting is that there is a delicate interplay between the \emph{size} of the graph and the \emph{time scale} on which the process is observed. For short times, the process behaves as if it lives on an infinite graph. For instance, many sparse graphs are locally tree-like and therefore the process behaves as if it evolves on an infinite Galton-Watson tree. For long times, however, the process sees that the graph is finite and exhibits a crossover in its behaviour. For instance, the voter  model, which will be the process of interest in the present paper, eventually reaches consensus on any finite connected graph, but the time at which it does depends on the size of the graph. Many such instances can be captured under the name of \emph{finite-systems scheme}, i.e., the challenge to identify how a finite truncation of a stochastic system behaves as both the time and the truncation level tend to infinity, properly tuned together (\cite{CG90}, \cite{CG94*}).  

Random processes on random graphs are part of the larger research area of random processes in \emph{random environment}, where the environment selects the random transition probabilities or transition rates. Another name is that of \emph{interacting particle systems} in random environment. For random process on lattices a more or less complete theory has been developed over the past four decades. The challenge is to extend this patrimony to random graphs. In the present paper we focus on the voter model on the regular random graph in the sparse regime. We track how the fraction of \emph{discordant edges} evolves over time in the limit as the size of the graph tends to infinity, and identify its scaling behaviour on \emph{three time scales}: short, moderate, and long. 

Voter models were introduced and studied in \cite{clifford:sudbury} and \cite{holley:liggett}. Voter models and their consensus times on finite graphs were analysed in \cite{donnelly1983finite} and \cite{C89}. The behaviour of voter models on random networks depends on the realisation of the network: \cite{sood:antal:redner} made various predictions for the expected consensus time on heterogeneous random networks (including power-law random graphs). More recently, \cite{fernley2019voter} study the asymptotics of the consensus time on inhomogeneous random graph models like the Chung-Lu model and the Norros-Reitu model. The expected consensus time for regular random graphs was studied in \cite{CFRmultiple}, in the discrete-time synchronous setting. As far as we know, there is no literature on the evolution of the number of discordant edges on random graphs. 


\subsection{Model and background}

Given a connected graph $G=(V,E)$, the voter model is the Markov process $(\eta_t)_{t \geq 0}$ on state space $\{0,1\}^V$, with $\eta_t = \{\eta_t(x)\colon\,x\in V\}$, where $\eta_t(x)$ represents the opinion at time $t$ of vertex $x$. Each vertex is equipped with an independent exponential clock that rings at rate one. After each ring of the clock, the vertex selects one of the neighbouring vertices uniformly at random and copies its opinion. (A formal description of this interacting particle system, its generator, and how its dynamics can be built up via the so-called graphical representation is postponed to Section \ref{sec:notation}.) 

As mentioned previously, while the voter dynamics has been widely studied on periodic lattices, periodic tori, and complete graphs, only recently its evolution on general connected graphs has been considered. In particular, as discussed below, if the underlying (random or non-random) graph has sufficiently nice properties, then it is possible to identify the time scale on which consensus takes place, and to determine how the process behaves on this time scale. In the present paper we specialise to the $d$-regular random graph ensemble with $d \ge 3$. In particular, we consider the sequence of random graphs $(G_{d,n}(\omega))_{n\in\N}$, with law denoted by $\P$, where each element $G_{d,n}(\omega) = (V,E(\omega))$ is a regular random graph of degree $d \ge 3$, consisting of $|V|=n$ vertices and $|E(\omega)|=m=dn/2$ edges, uniformly sampled from the set of simple $d$-regular graphs with $n$ vertices (to guarantee that $dn$ is even, for $d$ odd we restrict to $n$ even). These graphs are locally tree-like and have good expansion properties (see Section \ref{subsec:rrg_geometry}), lying within the realm of so-called \emph{mean-field geometries} (see Section \ref{sec:meanfeld}). Therefore the voter model on the $d$-regular random graph ensemble can be investigated in some depth, and refined statements clarifying how consensus is reached can be derived. In the present paper we analyse the evolution of the density of \emph{discordances}. In other words, denoting the set of discordant edges at time $t$ by
\begin{equation*}
	D^n_t = \big\{e=(x,y) \in E\colon\, \eta_{t}(x) \neq \eta_{t}(y)\big\},
\end{equation*}
we will study the fraction of discordant edges at time $t\ge 0$ given by
\begin{equation}\label{discord}
	\cD^n_t = \frac{|D^n_t|}{m}
\end{equation} 
in the limit as $n\to\infty$.

Note that on the complete graph the number of discordant edges is the product of the numbers of vertices carrying the two respective opinions. This is no longer the case on other geometries. Our results below for regular random graphs imply that \emph{homogenisation} occurs on time scale $n$, i.e.,  the product property emerges on the consensus time scale as the graph size grows, while different behaviour is observed on shorter time scales. 


\subsection{Voter model on mean-field geometries}
\label{sec:meanfeld}

$\mbox{}$

\medskip\noindent
$\bullet$ \emph{Voter model.}	
A classical model of relevance in population genetics is the voter model on the \emph{complete graph}, which is referred to as the \emph{Moran model} or the \emph{Fisher-Wright model} (depending on whether continuous-time asynchronous or discrete-time parallel updates are considered). The states $\{0,1\}$ represent two different alleles, each coding for a specific genetic trait, and the $n$ vertices of the complete graph $K_n$ represent individuals in a population of size $n$. In the limit as $n\to\infty$, the consensus time (to be interpreted as the extinction time of one of the two traits) is known to scale linearly in $n$. Furthermore, on time scale $n$ the fraction of the individuals of, say, type $1$ converges as a process to the so-called \emph{Fisher-Wright diffusion}. These results can be derived by analysing the backward genealogical progeny, which amounts to studying $n$ coalescing random walks evolving on the same graph, which in turn is related to the so-called \emph{Kingman coalescent} (we refer to \cite{Du08} for details). In Section \ref{sec:notation} we give a precise description of this \emph{duality} on arbitrary graphs.

These by now classical results have been recently extended to general \emph{mean-field geometries} under two main conditions that can be roughly described as follows: (I) the stationary distribution of the random walk on the random graph must be \emph{not too concentrated}; (II) \emph{fast-mixing} in the sense of an asymptotically vanishing ratio between the mixing time of a single random walk and the expected meeting time of two independent random walks starting from stationarity. 

\medskip\noindent
$\bullet$ \emph{Dual process.}
For what concerns the dual process of coalescing random walks, \cite{Oaop} computes the distributional limit of the coalescence time under the above mentioned assumptions. In particular, it follows from \cite[Theorem 1.2]{Oaop} that 
\begin{equation}
	\label{eq:Oliv}
	\lim_{n\to\infty}\frac{\Expect[\tau_{\rm coal}]}{\Expect[\tau^{\pi\otimes\pi}_{\rm meet}]}= 2,
\end{equation}
where $\tau_{\rm coal}$ is the coalescence time of $n$ random walks, each starting from a different vertex, and $\tau_{\rm meet}^{\pi\otimes\pi}$ is the meeting time of two random walks independently starting from stationarity. This offers insight into how the dual voter model behaves in mean-field like geometries. In particular, \eqref{eq:Oliv} allows us to translate the question of how long it takes for the voter model to achieve consensus to the question of controlling the meeting time of two random walks starting from stationarity. Indeed, as a consequence of this duality, it is immediate that the $n$-coalescence time $\tau_{\rm coal}$ stochastically dominates the above mentioned \emph{consensus time}, defined as
\begin{equation}
	\label{eq:def-consensus}
	\tau_{\rm cons}=\inf\{t\ge 0\colon\, \eta_t(x)=\eta_t(y), \text{ for all } x,y\in V \}.
\end{equation}
For what concerns the convergence to a Fisher-Wright diffusion after proper scaling, the recent work by \cite{CCC} considers the \emph{fraction of vertices in state $1$ at time $t$} in the voter model,
\begin{equation}
	\label{eq:def-frac-of-blue}
	\cB^{n}_t=\frac{1}{n}\sum_{x\in V}\eta_t(x), \qquad t \ge 0,
\end{equation}
under the above mentioned \emph{mean-field conditions} (see \cite[Theorem 2.2]{CCC} for a precise formulation). The main result states that, when time is sped up by a factor $\gamma_n \coloneqq \Expect[\tau_{\rm meet}^{\pi\otimes\pi}]$, the rescaled density of vertices in state $1$, i.e., $(\cB_{\gamma_n t}^{n})_{t \ge 0}$, converges as $n\to\infty$ to a Fisher-Wright diffusion in the Skorokhod topology.

\medskip\noindent
$\bullet$ \emph{Regular random graphs as mean-field geometries.}
Let us next discuss the implications of the results mentioned above within the specific framework of $d$-regular random graphs. The latter is a geometric setting that satisfies the aforementioned \emph{mean-field conditions} with high probability with respect to the law $\P$ of the environment. Indeed, it is well known that a typical realisation of the graph is connected with high probability as soon as $d\ge 3$. Due to the undirectedness of the edges, the corresponding stationary distribution is uniform over the vertex set. Furthermore, with high probability under $\P$, the mixing time is of order $\log n$ (see \cite{LS}), while the expected meeting time (see  \cite{chen-meeting}) satisfies
\begin{equation*}
	\frac{\Expect[\tau_{\rm meet}^{\pi\otimes\pi}]}{n} \overset{\mathbb{P}}{\longrightarrow} \frac{1}{2\theta_d}
\end{equation*}
with
\begin{equation}
	\label{eq:theta_d}
	\theta_d=\frac{d-2}{d-1}.
\end{equation}
Hence, by applying \eqref{eq:Oliv} to this setting, we find that
\begin{equation*}
	\frac{\Expect[\tau_{\rm coal}]}{n}\overset{\mathbb{P}}{\longrightarrow} \frac{1}{\theta_d}.
\end{equation*}
Moreover, for the voter model starting from independent Bernoulli opinions with parameter $u\in(0,1)$, by \cite{CCC} we also have that the fraction of opinions of type $1$, i.e., $\cB^n_t$ converges, after time is sped up by a factor $n$, to the Fisher-Wright diffusion $(\bar\cB_s)_{s \geq 0}$ given by the SDE
\begin{equation}
	\label{eq:dB}
	\dd \bar\cB_s = \sqrt{2\theta_d\bar\cB_s  (1-\bar\cB_s )}\,\dd W_s, \qquad \bar\cB_0  = u,
\end{equation}
where $(W_s)_{s \geq 0}$ denotes the standard Brownian motion (with generator $\tfrac12\Delta$). 

In conclusion, on $d$-regular ensembles consensus is reached on average on time scale $n$ and, as far as the law is concerned, on the time scale of the consensus the process can be well approximated by the Fisher-Wright diffusion in \eqref{eq:dB}. Interestingly, as pointed out by \cite{CCC}, the diffusion coefficient in \eqref{eq:dB} (which is referred to as the \emph{genetic variability} in a population genetics context) can be related to the fraction of discordant edges. This will be the starting point of our investigation, which is devoted to a deeper understanding of the evolution of the voter model beyond the consensus time scale $n$, via a detailed analysis of the discordant edges.
It is worth to notice that the number of discordant edges can be interpreted as the size of the interface between the two opinions, or, in other words, as the perimeter of the set of vertices having one of the two opinions. For this reason, tracking the fraction of discordant edges is interesting because it captures how consensus is reached via a gradual merging of the connected components with a single opinion. In a sense, the latter problem could be framed in the wide context of dynamic percolation: removing the discordant edges from the graph induces a certain number of connected components, each made by vertices having the same opinion. Nevertheless, in this work we will focus only on the analysis of the number of discordant edges and we leave the analysis of its consequences on the topology of the set of vertices with a given opinion for possible future work.


\subsection{Main theorems: evolution of discordances and stabilisation before consensus}

Before we present our main results, we introduce some notation. The symbols $\P$ and $\E$ will be reserved for the probability space of the $d$-regular random graph. The abbreviation $\whp$ refers to events that occur with  probability $\P$ tending to $1$ as $n\to\infty$.

Thanks to the duality between the $n$-vertex voter model and a system of $n$-coalescing random walks (to be described in more detail in Section \ref{sec:notation}), it is sufficient to look at a collection of $2m$ Poisson processes, each associated with an (oriented) edge of the graph, together with the initial assignment of the opinions. In this way, we can use the very same source of randomness to sample the evolution of the voter model, the dual system of coalescing random walks, and a system of $n$ independent random walks, each starting at a different vertex of the graph. For this reason we adopt the symbols $\Prob$ and $\Expect$ to refer to any of these three stochastic processes evolving on a quenched realisation of the graph. To distinguish between the random walks and the voter model, in the latter case we use a sub-index to refer to the initial distribution of the process. For instance, $\Prob_\xi$ denotes the law of the voter model starting from the configuration $\eta_0=\xi\in\{0,1\}^V$. Similarly, when we write $\Prob_u$, we refer to the voter model initialised by a product of Bernoulli random variables of mean $u\in(0,1)$, or to the solution of \eqref{eq:dB}. On the other hand, when considering a system of two or more independent or coalescing random walks, we simply use the symbols $\Prob$ and $\Expect$, and the starting positions of the random walks will be clear from the context (with the exception of Section \ref{sec:concentration}, where an \emph{ad-hoc} notation will be introduced). Finally, we write ${\rm Ber}(u)$ for a Bernoulli distribution of parameter $u$.  

We now present our main results. The first theorem identifies $\whp$ the first order asymptotics of the expected number of discordant edges on all time scales. 

\begin{theorem}[Expected density of discordances at all time scales]
	\label{theorem:expectation}
	Fix $d\geq 3$, and let $\theta_d$ be as in \eqref{eq:theta_d}. Fix $u\in(0,1)$ and, for $n\ge d+1$, consider the voter model on a $d$-regular random graph $G_{d,n}(\omega)$ with initial distribution $[{\rm Ber}(u)]^{\otimes V}$. Then, for any non-negative sequence of times $(t_n)_{n\in\N}$, such that the limit of $t_n$ and $t_n/n$ exist in $[0,\infty]$, the density of discordant edges in~\eqref{discord} satisfies
	\begin{equation}
		\label{AsyMean} 
		\left| {\bf E}_u\left[\cD^n_{t_n}\right]- 2u(1-u) f_d(t_n)\mathrm{e}^{-2 \theta_d \frac{t_n}{n} }\right| 
		\overset{\mathbb{P}}{\longrightarrow}0, 
	\end{equation}
	where  
	\begin{equation}
		\label{InfTreeMeet}
		f_d(t)={\bf P}^{\cT_d}(\tau^{x,y}_{\rm meet}> t)=\sum_{\kappa =0}^{\infty}
		\ee^{-2t}\frac{(2t)^\kappa}{\kappa!}\sum_{s>\lfloor \frac{\kappa-1}{2} \rfloor}\binom{2s}{s}\frac{1}{s+1} 
		\Big( \frac{1}{d} \Big)^{s+1}\Big( \frac{d-1}{d} \Big)^{s},
	\end{equation}
	with ${\bf P}^{\cT_d}$ denoting the law of two independent random walks on the infinite $d$-regular tree $\mathcal{T}_d$ starting from the endpoints of an edge $e=(x,y)$ in $\mathcal{T}_d$. 
\end{theorem}

Note that the limit is different depending on the time scale we are looking at:
\begin{itemize} 
	\item  
	({\bf Short time scale}) For $t_n=\Theta(1)$, the first order of the above expectation is given by $2u(1-u)f_d(t_n)$, and the behaviour is governed by the non-decreasing function captured in \eqref{InfTreeMeet}, representing the meeting time of two random walks starting from adjacent vertices on an infinite $d$-regular tree. See Figure \ref{fig:simulations}.
	\item  
	({\bf Moderate time scale}) For $t_n=\omega(1)\cap o(n)$, the discordances stabilise at a limiting plateau $2u(1-u) f_d(\infty)$ characterised by the probability that two adjacent random walks on an infinite tree never meet, i.e., $f_d(\infty)=\theta_d$.
	\item ({\bf Long time scale}) For $t_n=sn$, $s \in (0,\infty)$, the dual system coalesces and the voter model reaches consensus. The evolution of the density of opinions is approximated by the Fisher-Wright diffusion in \eqref{eq:dB}, and the expected density of discordances is characterised by the mean of the genetic variability ${\bf E}_u\left[\bar\cB_s (1-\bar\cB_s )\right]=2u(1-u)\theta_d\mathrm{e}^{-2 \theta_d s}$. See Figure \ref{fig:simulations2}.
	\item ({\bf Consensus}) For $t_{n}=\omega(n)$, the system has reached consensus and the limiting formula in \eqref{AsyMean} degenerates to zero.
\end{itemize}

\medskip
\begin{figure}[htbp]
	\centering
	\includegraphics[width=6.5cm]{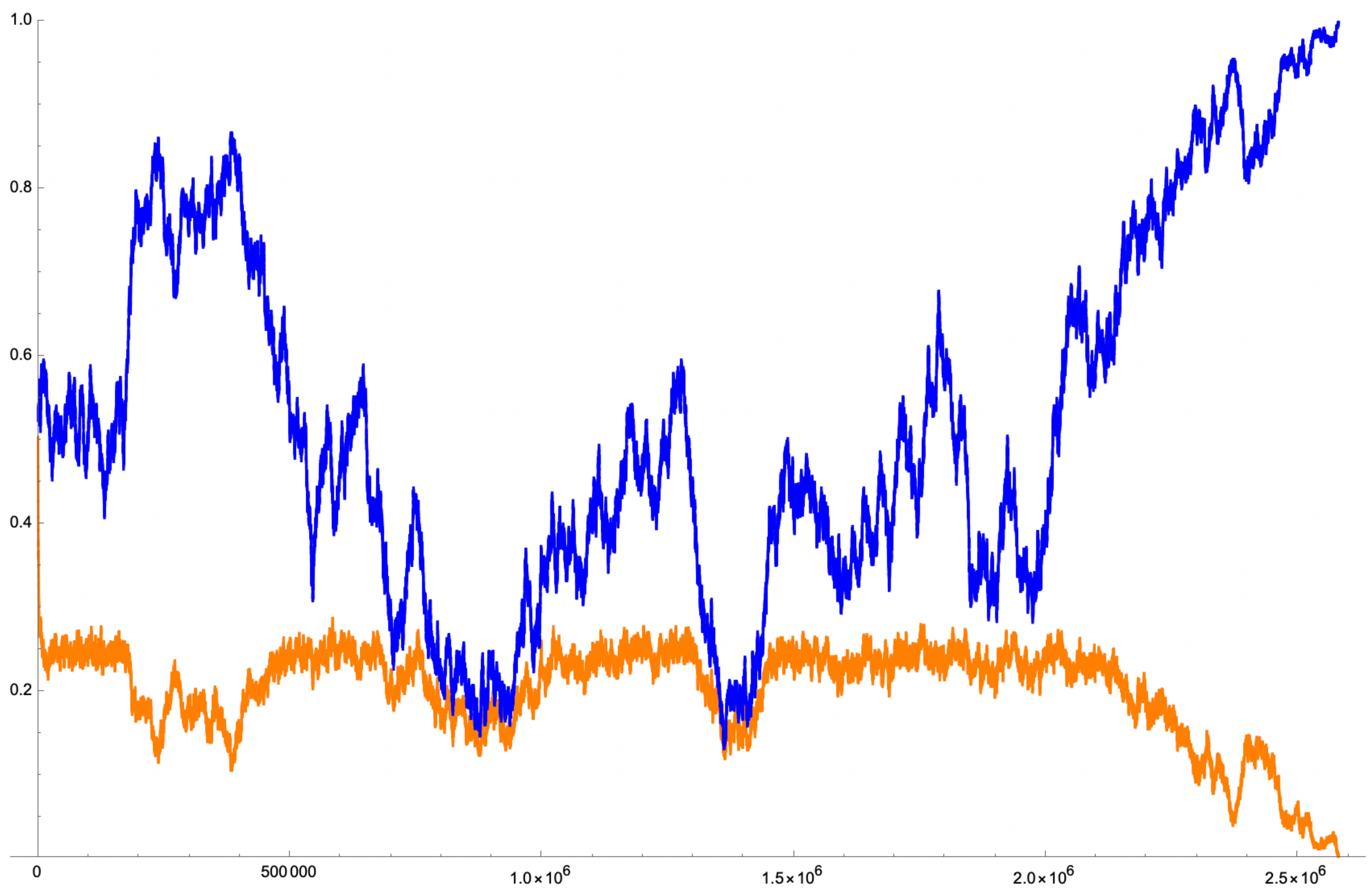}
	\hspace{1cm}
	\includegraphics[width=6.5cm]{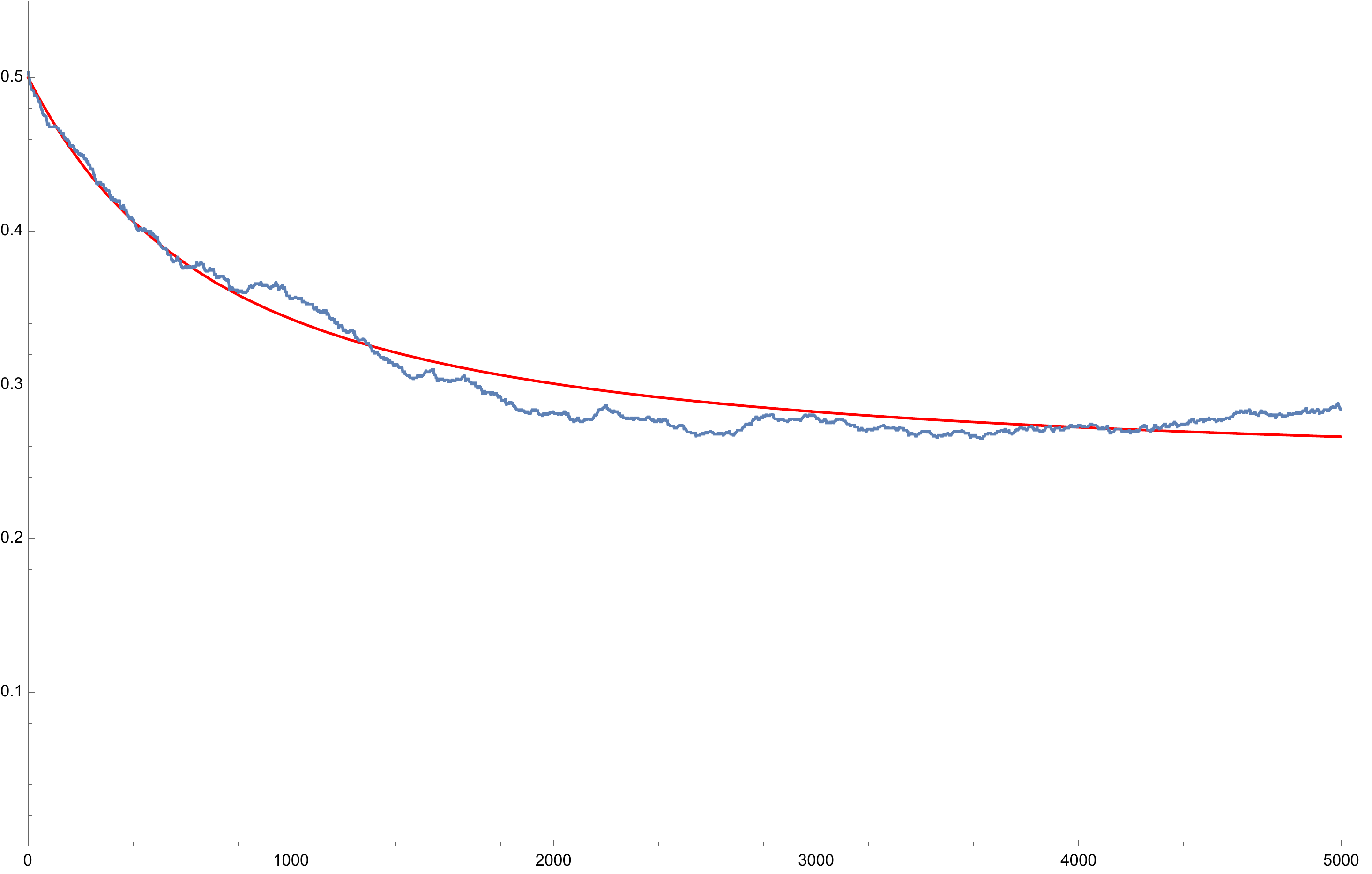}
	\vspace{0.3cm}
	\caption{The two plots show a single simulation of the voter model on a regular random graph of size $n=1000$, degree $d=3$ and initial density $u=0.5$. \emph{Left}: In blue the density of $1$-opinions up to consensus ($\tau_{\rm cons} \approx 2.6 \times 10^3$), in orange the density of discordant edges up to consensus. \emph{Right}: In blue the density of discordant edges up to time $t=5$ (corresponding to a zoom-in of the plot on the left), in red the function $t \mapsto 2u(1-u) f_d(t)$.}
	\label{fig:simulations}
\end{figure}

\begin{figure}[htbp]
	\centering
	\includegraphics[width=8.5cm]{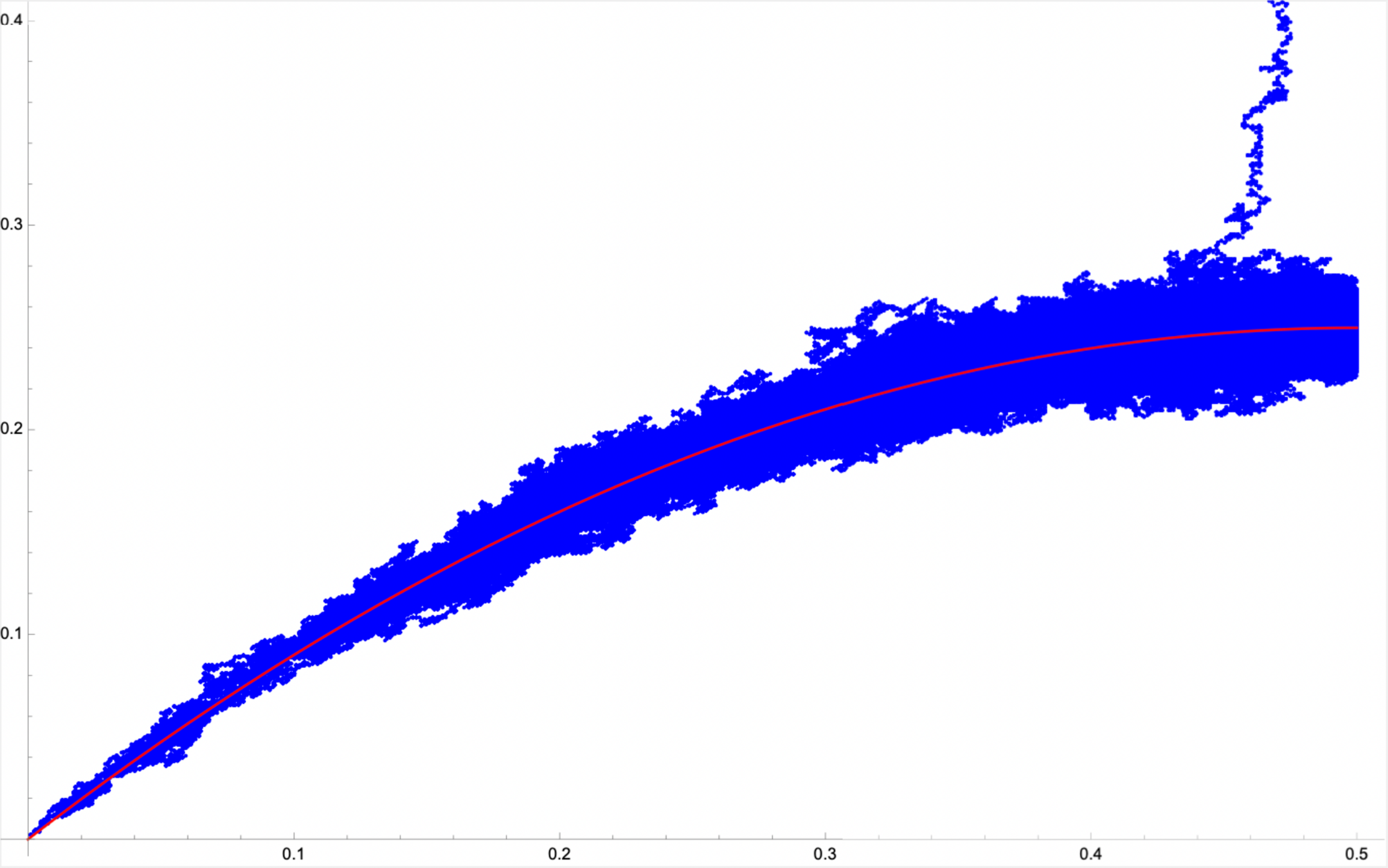}
	\vspace{0.3cm}
	\caption{Scatter plot in blue for the same simulation as in Figure \ref{fig:simulations}: the density of discordant edges versus the density of the minority opinion. The piece sticking out corresponds to short times. The curve in red is $x \mapsto x(1-x)$.}
	\label{fig:simulations2}
\end{figure}

\medskip
The second theorem characterises the process $\cD^n$ in~\eqref{discord} beyond its expectation. In particular, we show that on time scale $o(n)$ the density of discordances concentrates around the expectation, while on time scale $\Theta(n)$ it behaves as a deterministic function of the Fisher-Wright diffusion.

\begin{theorem}[Beyond expectation]
	\label{thm:beyondthemean}
	Consider the voter model on a $d$-regular random graph $G_{d,n}(\omega)$ with $d\geq 3$.
	\begin{enumerate}
		\item \textbf{Concentration before coalescence.} Let $t_n$ be such that $t_n/n \to 0$. Then, for every $\varepsilon>0$, 
		\begin{equation*}
			\sup_{\xi\in\{0,1 \}^V} \Prob_\xi\left(\left|\cD^n_{t_n}-\Expect_\xi[\cD_{t_n}^n] \right|>\varepsilon \right)\overset{\P}{\longrightarrow}0.
		\end{equation*}
		\item \textbf{Discordance on consensus time scale.} Let $t_n$ be such that $t_n/n \to s \in (0,\infty)$. Then, for every $u\in(0,1)$,
		\begin{equation*}
			\sup_{x\in[0,1]} \left| \Prob_u\left(\cD^n_{t_n}\le x \right) - \Prob_u\left(2\theta_d\bar\cB_{2s\theta_d} (1-\bar\cB_{2s\theta_d} )\le x \right) \right| 
			\overset{\P}{\longrightarrow} 0,
		\end{equation*}
		where $(\bar\cB_s )_{s\ge0}$ is the solution of \eqref{eq:dB}.
	\end{enumerate}
\end{theorem}

The third and last theorem is a strengthening of Theorem \ref{thm:beyondthemean}.1. More precisely, we show that as soon as we focus on the process for times that are polynomially smaller than the consensus time, the concentration around the mean is actually uniform in time.  

\begin{theorem}[Uniform concentration on moderate time scale]
	\label{th:unif-concentration}
	Consider the voter model on a $d$-regular random graph $G_{d,n}(\omega)$. Then, for every $u\in(0,1)$ and $\delta,\varepsilon>0$,
	\begin{equation*}
		\Prob_u \Big( \sup_{0 \le t\le n^{1-\delta}}\big|\cD_t^n- \Expect_u[\cD_t^n]\big|>\varepsilon \Big) \overset{\P}{\longrightarrow} 0.
	\end{equation*}
\end{theorem}


\subsection{Open problems}
\label{sec:openproblems}

We point out two open problems. 
\begin{itemize}
	\item
	Theorem~\ref{th:unif-concentration} says that the concentration of the fraction of discordant edges is uniform up to times $n^{1-\delta}$, for any $\delta>0$. We expect that Theorem~\ref{th:unif-concentration} can be strengthened to the statement that for every $t_n$ such that $t_n/n \to 0$ and every $C_n \to \infty$,
	\begin{equation}
		\label{eq:conj}
		\Prob\big(\left|\cD^n_{t_n}-\Expect[\cD^n_{t_n}]\right| > C_n \sqrt{t_n/n}\,\big)\overset{\P}{\longrightarrow} 0.
	\end{equation}
	Note that, thanks to Azuma inequality, the concentration in \eqref{eq:conj} holds for the quantity $\cB_t^n$ in \eqref{eq:def-frac-of-blue}, because this is a martingale. Note further that the concentration in Theorem~\ref{th:unif-concentration} follows from \eqref{eq:conj} and a union bound. For now, \eqref{eq:conj} is beyond our reach.
	\item
	We expect that Theorems~\ref{theorem:expectation}--\ref{th:unif-concentration} can be extended to \emph{non-regular} sparse random graphs. We do not have a conjecture on how the function $f_d$ and the diffusion constant $\theta_d$ modify in this more general setting.
\end{itemize}


\subsection{Description of techniques}

All our proofs are based on the classical notion of duality between the voter model and a collection of coalescent random walks. In particular, on time scales $o(\log n)$, i.e., below the typical distance between two vertices, the analysis can be carried out by coupling a system of two random walks, starting at adjacent vertices and evolving on the $d$-regular random graph, with two random walks on the $d$-regular tree. Because the tree is regular, the distance of the two random walks can be viewed as the distance to the origin of a single biased random walk on $\N_0$. Note that the same does not hold when the tree is not regular. 

Clearly, in order to analyse the process on time scales $\Theta(\log n)$, i.e., on the typical distance between two vertices, the coupling argument must be combined with a finer control of the two random walks on the regular random graph. To this aim, we exploit the strategy developed in \cite{CFRmultiple}, which uses the so-called \emph{First Visit Time Lemma} (see Theorem \ref{t:fvtl}). Such control, together with a first-moment argument, is enough to compute the evolution of the expected number of discordant on every time scale. 

On the other hand, in order to obtain the concentration result in Theorem \ref{th:unif-concentration}, a much deeper analysis is required. Roughly, in order to have proper control on the correlations between the edges constituting $\cD_t^n$, we must analyse a dual system of random walks whose number grows with $n$. Exploiting a classical result by \cite{AB1}, we derive upper bounds for the number of meetings of a poly-logarithmic number of independent random walks evolving on the random graph for a time $n^{1-o(1)}$. Such a bound will be exploited in the forthcoming Proposition \ref{prop:unif-concentration} to deduce an upper bound for the deviation from the mean that is exponentially small in $n$ and uniform in time. This upper bound can in turn be translated to the result in Theorem \ref{th:unif-concentration} by taking a union bound.


\subsection{Outline}

Section~\ref{sec:notation} lists definitions and notations, recalls the graphical construction and duality, states a key lemma, and collects a few basic facts about regular random graphs and random walks on such geometries. Theorems~\ref{theorem:expectation}--\ref{th:unif-concentration} are proved in Sections~\ref{sec:exp}--\ref{sec:concentration}, respectively. Appendix~\ref{app:aux} contains two auxiliary facts for c\`adl\`ag processes.

\subsection{Acknowledgements} This research was supported by the European Union’s Horizon 2020 research and innovation programme under the Marie Skłodowska-Curie grant agreement no. 101034253, and by the NWO Gravitation project NETWORKS under grant no. 024.002.003. RB has counted on the support of ``Conselho Nacional de Desenvolvimento Científico e Tecnológico - CNPq'' grants ``Projeto Universal'' (402952/2023-5) and ``Produtividade em Pesquisa'' (308018/2022-2). MQ acknowledges financial support by the German Research Foundation (project number 444084038, priority program SPP2265). The authors thank Federico Capannoli for his helpful suggestions. Finally, FdH and MQ gratefully acknowledge hospitality at the Simons Institute for the Theory of Computing, Berkeley, USA, as part of the semester program ``Graph Limits and Processes on Networks: From Epidemics to Misinformation'' (GLPN22) in the Fall of 2022.

\section{Notation and preparation}
\label{sec:notation}

In this section we properly introduce the voter model on general graphs, and collect some results that will be needed along the way. In Section~\ref{subsec:voter_model} we introduce the graphical construction for the voter model and the associated dual process, known as \emph{coalescing random walks}. Section~\ref{subsec:FVTL} contains the \emph{First Visit Time Lemma}, which has been introduced by \cite{CF1}. Section~\ref{subsec:rrg_geometry} collects some useful facts about the geometry of $d$-regular random graphs and the behaviour of random walks on them. In what follows we drop the upper index $n$ to lighten the notation. 


\subsection{The voter model and coalescing random walks}
\label{subsec:voter_model}

Let $G(V,E)$ be a connected (possibly infinite and locally-finite) undirected graph. For each $x\in V$, let $d_x$ denote the degree of vertex $x$. The voter model on $G$ is defined as the interacting particle system with state space $\{0,1\}^V$ and generator $\mathcal{L}$ acting on local functions $f\colon\,V \to \mathbb{R}$ as
\begin{equation*}
	(\mathcal{L}f)(\eta) =\sum_{x \in V} \sum_{y \sim x} \frac{1}{d_x}  \big(f(\eta^{x\leftarrow y}) - f(\eta)\big),
	\qquad \eta\in\{0,1\}^V,
\end{equation*}
where $\eta^{x\leftarrow y}$ is obtained from $\eta$ as
\begin{equation*}
	\eta^{x\leftarrow y}(z) = \begin{cases}
		\eta(y), & \quad \text{if } z=x, \\
		\eta(z), & \quad \text{otherwise}.
	\end{cases}
\end{equation*}
In this system, agents (represented by the vertices of $G$) start with binary opinions and update these at rate one by copying the opinion of a uniformly chosen neighbour. For $t \geq 0$ and $x \in V$, denote by $\eta_{t}(x) \in \{0,1\}$ the opinion of vertex $x$ at time $t$.

When $G$ is finite and regular, i.e., $d_x=d$ for every $x\in V$, the number of vertices that have opinion $1$,
\begin{equation*}
	B_t = \sum_{x \in V} \eta_t(x),
\end{equation*}
is a martingale. Since $B_t$ is bounded by $|V|$, it converges, which implies that the voter model eventually fixates on a constant configuration, composed entirely of $0$'s or $1$'s, denoted by $\bar{0}$ and $\bar{1}$, respectively. Define the consensus time of $G$ as the stopping time
\begin{equation*}
	\tau_{\rm cons} = \inf \big\{t \geq 0\colon\, \eta_{t} \in \{\bar{0}, \bar{1}\} \big\}.
\end{equation*}

Let us next introduce the \emph{graphical construction} of the voter model. Associate to every oriented edge $\vec{e} = (x,y)$ a Poisson point process $\mathscr{P}_{\vec{e}}$ on $\mathbb{R}$ with intensity $1/d_x$. When a clock from $\vec{e}=(x,y)$ rings, vertex $x$ receives the opinion of vertex $y$. This constriction allows us to track the joint evolution of all the opinions via time duality. In order to determine $\eta_t(x)$, the state of a vertex $x$ at time $t$, we start from the space time point $(x,t)$, and go backwards in time, crossing every clock ring from an edge that reaches the current position of the path (see Figure \ref{fig:graphical_construction}).

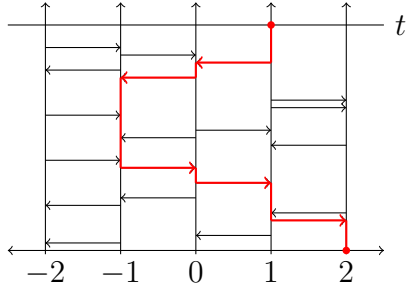
\begin{figure}[htbp]
	\begin{tikzpicture}
		\draw[<->] (-2.5,-0)--(2.5,0);
		\draw(-2.5,3)--(2.5,3);
		\node[right] at (2.5,3) {$t$};
		\foreach \x in {-2, -1, 0, 1, 2}{
			\draw[->](\x,-0.05)--(\x,3.3);
			\node[below] at (\x,0) {$\x$};
		}
		\foreach \y in {0.1, 0.6, 2.4}{
			\draw[<-](-2,\y)--(-1,\y);
		}
		\foreach \y in {0.7, 1.5, 2.3}{
			\draw[<-](-1,\y)--(0,\y);
		}
		\foreach \y in {0.2, 2.5}{
			\draw[<-](0,\y)--(1,\y);
		}
		\foreach \y in {0.5, 1.4}{
			\draw[<-](1,\y)--(2,\y);
		}
		\foreach \y in {1.2, 1.8, 2.7}{
			\draw[->](-2,\y)--(-1,\y);
		}
		\foreach \y in {1.1, 2.6}{
			\draw[->](-1,\y)--(0,\y);
		}
		\foreach \y in {0.9, 1.6}{
			\draw[->](0,\y)--(1,\y);
		}
		\foreach \y in {0.4, 1.9, 2}{
			\draw[->](1,\y)--(2,\y);
		}
		\fill[red] (1,3) circle (0.05);
		\draw[red, thick] (1,3) --(1,2.5);
		\draw[<-, red, thick](0,2.5)--(1,2.5);
		\draw[red, thick] (0,2.5) --(0,2.3);
		\draw[<-, red, thick](-1,2.3)--(0,2.3);
		\draw[red, thick] (-1,2.3) --(-1,1.1);
		\draw[->, red, thick](-1,1.1)--(0,1.1);
		\draw[red, thick] (0,1.1) --(0,0.9);
		\draw[->, red, thick](0,0.9)--(1,0.9);
		\draw[red, thick](1,0.9)--(1,0.4);
		\draw[->, red, thick](1,0.4)--(2,0.4);
		\draw[red, thick](2,0.4)--(2,0);
		\fill[red] (2,0) circle (0.05);
	\end{tikzpicture}
	\vspace{0.3cm}
	\caption{The graphical construction of the voter model on the integer lattice $\Z$. The red path allows us to determine that $\eta_t(1)$ is equal to $\eta_0(2)$.}
	\label{fig:graphical_construction}
\end{figure}

The process that tracks the origin of the opinions can be viewed as a collection of \emph{coalescing random walks}. To define this process we assume that $G$ is finite and consider a family $\{(X^v_t)_{t \geq 0}\}_{v \in V}$ of independent rate-one continuous-time random walks such that $X^v_0 = v$ almost surely. In what follows we identify $V$ with $[n] = \{1,\ldots,n\}$. When two walks meet at the same vertex they coalesce into a single walk (this can be seen immediately by the graphical representation in \ref{fig:graphical_construction}.)

Because $G$ is finite, any two independent random walks meet in finite time. Consequently, $\tau_{\rm coal}$, the first time at which all the random walks sit at the same vertex called the \emph{coalescence time}, satisfies $\tau_{\rm coal}<\infty$ almost surely. In particular, the above mentioned duality implies that $\tau_{\rm cons} \leq \tau_{\rm coal}$, and therefore any upper bound on the coalescence time immediately yields an upper bound on the consensus time of the voter model.

Given $x,y \in [n]$, let
\begin{equation}
	\label{TxyMeetDef}
	\tau_{\rm meet}^{x,y} = \inf\{t \geq 0\colon\, X_{t}^{x}=X_{t}^{y}\}
\end{equation}
denote the \emph{meeting time} of two independent continuous-time random walks starting from $x$ and $y$, respectively. More generally, if the two random walks are initialised randomly and independently according to two distributions $\mu$ and $\nu$ on $[n]$, then we denote their meeting time by $\tau_{\rm meet}^{\mu \otimes \nu}$. We note that the extension to infinite graphs $G$ follows from standard estimates on the probability that the process can be approximated by a finite dynamics in finite portions of space during a bounded amount of time.


\subsection{The First Visit Time Lemma}
\label{subsec:FVTL}

In this section we state the \emph{First Visit Time Lemma} (FVTL), which was introduced by \cite{CF1} and recently simplified by \cite{MQS}. This lemma contains estimates on the time a random walk takes to reach a distinguished site when starting from the stationary distribution. 

\begin{remark}\label{remark:dt}
	Since the FVTL is defined in discrete time, in what follows it will be easier to derive the results in discrete time and then translate them to continuous time. For this reason, when considering two random walks in discrete time, where at each step the random walk that moves is selected uniformly at random, we use the notation $\Prob^{\rm dt}$ to refer to their law.
\end{remark}

\begin{theorem}[First Visit Time Lemma; \cite{MQS}]
	\label{t:fvtl}
	Consider a sequence $((X^{(N)}_t)_{t\ge 0})_{N\in\N}$ of ergodic Markov chains, each living on a state space $\Omega_N$ of size $N$ with transition matrix $Q_N$ and stationary distribution $\mu_N$, and a sequence of target states $(x_N)_{N\in\N}$, with $x_N\in\Omega_N$. Define the mixing time
	\begin{equation}
		\label{eq:tmix}
		T_{\rm mix}^{(N)} = \inf\left\{t \geq 0\colon\,\max_{y_N,z_N \in \Omega_N} |Q_N^t(y_N,z_N)-\mu_N(z_N)| \leq \frac{1}{N^3} \right\},
	\end{equation}
	and assume that
	\begin{enumerate}
		\item $\lim_{N\to\infty} N^2\min_{y_N \in \Omega_N} \mu_{N}( y_N )=\infty$.
		\item $\lim_{N\to\infty} T_{\rm mix}^{(N)}\max_{ y_N \in \Omega_N} \mu_{N}(y_N)=0$.
		\item There exists a unique \emph{quasi-stationary distribution} associated to $x_N$, i.e., for all $N$ large enough the sub-matrix $[Q_N]_{x_{N}}$ in which the row and column indexed by $x_{N}$ have been erased is \emph{irreducible}. See also \cite[Remark 3.2 and Section 4.4]{QS22}.
	\end{enumerate}
	Then there exists a sequence $(\lambda_{x_N}^{(N)})_{N \in \N}$ such that, for every sequence $(T^{(N)})_{N\in\N}$ such that $T^{(N)}\ge T_{\rm mix}^{(N)}$ and such that {\rm (2)} is satisfied with $T^{(N)}$ instead of $T^{(N)}_{\rm mix}$,
	\begin{equation*}
		\lim_{N\to\infty} \left| \frac{\lambda^{(N)}_{x_N}}{\frac{\mu_{N}(x_N)}{R_{N}(x_N)}} -1 \right| = 0, 
		\qquad R_{N}(x_N) \coloneqq \sum_{s=0}^{T^{(N)}} Q_{N}^s(x_N, x_N),
	\end{equation*}
	and
	\begin{equation*}
		\lim_{N\to\infty} \sup_{t \geq 0} \left| \frac{\Pr_{\mu_{N}}(\tau_{x_N}>t)}{(1-\lambda^{(N)}_{x_N})^t } -1 \right| = 0,
	\end{equation*}
	with $\Pr_{\mu_{N}}$ the path measure of the Markov chain starting from $\mu_N$, and $\tau_{x_{N}}$ is the hitting time of $x_{N} \in \Omega_{N}$
	$$
	\tau_{x_N}\coloneqq\inf\{t\ge 0\colon\, X_t^{(N)}=x_N \}. 
	$$
\end{theorem}

In words, the FVTL says that, under the assumptions stated, the hitting time of a state $x_N$ is well approximated by a geometric random variable with mean $\frac{R_N(x)}{\mu_N(x)}$.


\subsection{The geometry of regular random graphs}
\label{subsec:rrg_geometry}

In this section we collect a number of definitions and asymptotic results about the \emph{typical geometry} of regular random graphs. We start by introducing notation and conclude by stating some basic facts regarding typical $d$-regular random graphs. 

Given a connected graph $G=(V,E)$ and two vertices $x,y \in V$, we denote by ${\rm dist}(x,y)$ the distance between $x$ and $y$, given by the number of edges of the shortest path joining $x$ and $y$. For a vertex $x \in V$ and a radius $h \geq 1$, we denote by $\mathfrak{B}_{x}(h)$ the ball of radius $h$ centered at $x$, seen as a subgraph of $G$. The tree excess of a graph quantifies how far away from a tree a given connected graph is.

\begin{definition}[Tree excess]
	Given a connected graph $G=(V,E)$, the \emph{tree excess} of $G$, denoted by $\tx(G)$, is the minimum number of edges that must be removed from $G$ in order to obtain a tree. This can also be written as
	\begin{equation*}
		\tx(G) = |E|-|V|-1.
	\end{equation*}
\end{definition}

Next, we introduce the notion of \emph{locally tree-like vertices and edges}. 	

\begin{definition}[Locally tree-like vertex]
	Given a graph $G=(V,E)$, we say that a vertex $x \in V$ is locally tree-like up to distance $\ell \geq 1$, denoted by $x \in {\rm LTL}(\ell)$, if the subgraph induced by the vertices at distance at most $\ell$ from $x$, i.e., $\mathfrak{B}_x(\ell)$, is a tree.
\end{definition}

\begin{definition}[Locally tree-like edge]
	Given a graph $G=(V,E)$ and an edge $e \in E$, we say that $e = \{x,y\}$ is locally tree-like up to distance $\ell \geq 1$, denoted by $e \in {\rm LTLE}(\ell)$, if the subgraph induced by the vertices at distance at most $\ell$ from $x$ or $y$, i.e., $\mathfrak{B}_x(\ell) \cup \mathfrak{B}_y(\ell)$, after removal of the edge $e$, is composed of two disjoint trees of depth at most $\ell$.
\end{definition}

For $d \geq 3$ and $n>d$ such that $nd$ is even, there exists at least one graph with $n$ vertices and constant degree equal to $d$. In particular, for this range of parameters the $d$-regular random graph of size $n$, drawn from the uniform distribution on the set of all graphs with $n$ vertices and constant degree equal to $d$, is well defined. Denote by $m=dn/2$ the number of edges of this random graph. The following proposition collects properties of the $d$-regular random graph that will be needed in the rest of the paper.

\begin{proposition}[Geometry of regular random graphs and random walks; \cite{LS}]
	\label{prop:geometry-rrg}
	Fix $d \geq 3$, and let $G_{d,n}(\omega)$ denote the $d$-regular random graph of size $n$. Then, $\whp$,
	\begin{enumerate}
		\item $G$ is connected.
		\item For every $x\in [n]$, $\tx(\mathfrak{B}_x(\tfrac{1}{5}\log_d n))\leq 1$.
		\item $|{\rm LTLE}(\tfrac{1}{5} \log_d n)| = m-o(n)$.
		\item There exists a constant $C=C_d>1$ such that, for $t\ge C\log n$ and every $\varepsilon>0$,
		\begin{equation}\label{TV}
			d_{\rm TV}(t):=\max_{x\in [n]}\| P^t(x,\cdot)-\pi \|_{\rm TV}\le \varepsilon, 
		\end{equation} 
		where $P^t(x,\cdot)$ denotes the law at time $t$ of the random walk starting from $x$. In particular, $\displaystyle\lim_{n\to\infty}d_{\rm TV}(t_n)=0$ for $t_n$ such that $\displaystyle\lim_{n\to\infty}(\log n)/t_n=0$.
	\end{enumerate}
\end{proposition}

\section{Computation of the expectation}
\label{sec:exp}

In this section we prove Theorem \ref{theorem:expectation}, in particular, we provide a $\whp$ first-order estimate of the expected density of discordant edges at time $t=t_n$, for every choice of the sequence $(t_n)_{n\in\N}$. In Section~\ref{subsec:voter_trees} we collect some basic results for the voter model on the infinite $d$-regular tree $\cT_d$, the \emph{local approximation} of the $d$-regular random graph. We show that the density of discordant edges at time $t$ on $\cT_d$ with ${\rm Ber}(u)$ initialisation behaves as $2u(1-u)(1-f_d(t))$. This will be done by exploiting duality: we reduce the problem to the analysis of a pair of coalescing random walks starting at the two extremes of an edge of the tree. In Section \ref{subsec:1} we show that a classical coupling argument suffices to show that the approximation holds for every sequence $t = (t_{n})_{n \in \N}$ such that $t/\log n \to 0$. When $t$ starts to become comparable with the typical distance of the graph, i.e., $t_{n}=\Theta(\log n)$, the coupling argument fails and we need a refined analysis of the process to compute the expected density. This scenario will be handled in Section \ref{subsec:2}. The idea is to track a pair of random walks starting from the two extremes of a typical (hence, \emph{locally tree-like}) edge up to their first meeting. Using the First Visit Time Lemma, we show that, under the event that the two random walks do not meet after a short time, the time of their first meeting is well approximated by exponential random variable with rate $\theta_d^{-1}n$.


\subsection{Discordant edges on the regular tree}
\label{subsec:voter_trees}

Let $\mathcal{T}_d$ denote the \emph{infinite $d$-regular tree} with $d\geq 3$. Our first lemma concerns the probability that two independent random walks, starting at distance one from each other, do not meet within a time $t$. Recall that $\Prob^{\cT_{d}}$ denotes the law of two independent random walks starting from the end vertices of an edge in $\cT_{d}$.

\begin{lemma}[Meeting on a regular tree]
	\label{lemma:gambler}
	Let $x,y \in \mathcal{T}_d$ with ${\rm dist}(x,y)=1$. Then
	\begin{equation}
		\label{eq:small_meeting_time}
		1-f_d(t) =\Prob^{\cT_d}\big(\tau^{x,y}_{\rm meet} \leq t \big)= \sum_{\kappa =0}^{\infty}
		\ee^{-2t}\frac{(2t)^\kappa}{\kappa!}\sum_{s=0}^{\lfloor \frac{\kappa-1}{2} \rfloor}\binom{2s}{s}\frac{1}{s+1} 
		\Big( \frac{1}{d} \Big)^{s+1}\Big( \frac{d-1}{d} \Big)^{s}.
	\end{equation}
	From this it follows that
	\begin{equation}
		\label{eq:lemma_gambler}
		\Prob^{\cT_d}(\tau^{x,y}_{\rm meet}=\infty) = \frac{d-2}{d-1} = \theta_d.
	\end{equation}
\end{lemma}

\begin{proof}
	Let $(X_t)_{t \geq 0}$ and $(Y_t)_{t \geq 0}$ be two independent continuous-time random walks starting from $x \sim y$, respectively. Put $Z_t = {\rm dist}(X_t,Y_t)$ and note that $(Z_t)_{t \geq 0}$ is a continuous-time biased random walk on $\N_0$ that starts at $1$, jumps at rate $2$, and has jump distribution given by
	\begin{equation}
		\label{eq:jump_distribution}
		p(z,z+1) = 1-p(z,z-1) = \frac{d-1}{d}, \qquad z \in \N. 
	\end{equation}
	The claim in \eqref{eq:small_meeting_time} follows from the construction of the random walk $(Z_t)_{t \geq 0}$ via a Poisson process of rate $2$ for the jump times, together with a skeleton chain given by the discrete-time biased random walk with jump distribution given by \eqref{eq:jump_distribution}, and explicit expressions for hitting times of discrete-time random walks based on path-counting arguments, see~\cite[Theorem~5.7.7]{durrett_book}.
	
	To get \eqref{eq:lemma_gambler} it is enough to note that, by the standard Gambler's ruin argument,
	\begin{equation*}
		\Prob^{\cT_d}\big(\tau^{x,y}_{\rm meet} < \infty\big) = \Pr\big(Z_t=0 \text{ for some } t \geq 0 \mid Z_0=1 \big) 
		= \frac{p(0,-1)}{p(0,1)} = \frac{1}{d-1},
	\end{equation*}
	from which the claim follows.
\end{proof}

\begin{remark}
	Note that \eqref{eq:lemma_gambler} can be derived from \eqref{eq:small_meeting_time} by direct computation. Indeed, clearly $1-f_d(0)=0$. The generating function $c(x) = \sum_{k=0}^\infty C_k x^k$ for the Catalan numbers $C_k=\binom{2k}{k}\frac{1}{k+1}$ solves the quadratic equation $c(x) = 1 + xc(x)^2$. Since $C_0=1$, we have 
	\begin{equation*}
		c(x)=\frac{1-\sqrt{1-4x}}{2x}.
	\end{equation*} 
	It follows that 
	\begin{equation*}
		1-f_d(\infty) = \frac{1}{d}\, c\left(\frac{d-1}{d^2}\right) = \frac{1}{d}\,\frac{1-\frac{d-2}{d}}{2\:\frac{d-1}{d^2}}=\frac{1}{d-1}=1-\theta_d.
	\end{equation*}
	Indeed, the first equality follows from the fact that in the limit as $t \to\infty$ the terms in \eqref{eq:small_meeting_time} with $\kappa\to\infty$ become dominant.
\end{remark}

\begin{lemma}[Density of discordant edges]
	\label{lemma:voter_model_tree}
	Consider the voter model on $\mathcal{T}_d$ with initial density $u$. Then, for any given edge $e$ and any time $t \geq 0$,
	\begin{equation*}
		\Prob^{\cT_d}_u \big( e \in D_t \big) = 2u(1-u)f_d(t),
	\end{equation*}
	where $f_d(t)$ is given by \eqref{eq:small_meeting_time}.
\end{lemma}

\begin{proof}
	Put $e=\{x,y\}$. Note that in order to have $e \in D_t$, by duality, the (backward) random walks starting from $x$ and $y$ at time $t$ do not coalesce before time $0$ and end up at vertices with distinct initial opinions. The initial opinions are independent of the trajectories of these random walks. The claim follows by noting that $f_d(t)$ is the probability that the two random walks do not meet in time $t$, and $2u(1-u)$ is the probability that two distinct vertices have different initial opinions.
\end{proof}	


\subsection{Expectation for small times}
\label{subsec:1}

We start this section with a general bound that does not require the graph $G$ to be random nor the size of $G$ to grow to infinity.

\begin{lemma}[Discordant LTLEs]
	\label{lemma:general_graph}
	Let $G = (V,E)$ be any $d$-regular graph and let $e= \{x,y\} \in E$ be ${\rm LTLE}(\ell)$ for some $\ell$. Consider the voter model on $G$ with initial density $u$. Then, for any time $T>0$,
	\begin{equation*}
		\sup_{t\in[0,T]} \left|\Prob_u(e\in D_t )-2u(1-u)f_d(t)\right| \leq \frac{4T}{\ell}.
	\end{equation*}
\end{lemma}

\begin{proof}
	Note first that the term $2u(1-u)f_d(t)$ corresponds to the probability that a given edge in $\mathcal{T}_{d}$ is discordant at time $t$. In particular, if we consider the dual system with coalescing random walks, then the probabilities we are interested in are the same, provided the random walks do not leave the tree-like neighbourhood of $e$. Therefore consider the event
	\begin{equation*}
		\cE_T = \big\{ \text{the number of jumps the random walks perform before time } T \text{ is bounded by } \ell \big\},
	\end{equation*}
	and observe that, by the Markov inequality,
	\begin{equation*}
		\Prob\big(\cE_T^{c} \big) = \Prob^{\cT_d}\big(\cE_T^{c}\big) \leq \Prob^{\cT_d}\big(X \geq \ell\big) \leq \frac{2T}{\ell},
	\end{equation*}
	where $X \sim {\rm Poisson}(2T)$ stochastically dominates the number of jumps that both random walks make up to time $T$. By construction, for $t \leq T$ we can write
	\begin{equation*}
		\Prob_u\big(e \in D_t , \cE_T\big) = \Prob^{\cT_d}_u\big(e \in D_t, \cE_T \big),
	\end{equation*}
	from which we can estimate
	\begin{equation*}
		\begin{split}
			&\left|\Prob_u(e\in D_t )-2u(1-u)f_d(t)\right| 
			= \Big|\Prob_u\big(e \in D_t\big) - \Prob^{\cT_d}_u\big(e \in D_t \big)\Big| \\
			& \leq \Big|\Prob_u\big(e\in D_t, \cE_T \big) - \Prob_u^{\cT_d}\big(e \in D_t, \cE_T\big)\Big| + 2\Prob^{\cT_d}\big(\cE_T^{c}\big) 
			\leq \frac{4T}{\ell},
		\end{split}
	\end{equation*}
	which settles the claim.
\end{proof}

\begin{proposition}[Short time average]
	\label{prop:mean_linear_time}
	Consider the voter model on a regular random graph $G_{d,n}(\omega)$ with initial density $u\in(0,1)$. Then, for any time $t_n$ satisfying $t_n/\log n\to0$,
	\begin{equation}
		\label{eq:prop_mean_linear_time}
		\left|\Expect_u[\cD_{t_n}]-2u(1-u)f_d(t_n)\right| \overset{\P}{\longrightarrow} 0.
	\end{equation}
\end{proposition}

\begin{proof}
	Note that for every realisation of $G$,
	\begin{equation*}
		\Expect_u[|D_{t_n}|]=\sum_{e\in E} \Prob_u(e\in D^n_{t_n}).
	\end{equation*}
	Call $E_\star\subset E$ the set of ${\rm LTLE}\big(\frac{1}{5}\log_d n\big)$ edges. Then, by Lemma \ref{lemma:general_graph} and our assumption on $t_n$,
	\begin{equation*}
		\begin{split}
			\Big|\Expect_u[\cD_{t_n}] - 2u(1-u)f_d(t_{n}) \big| & \leq \Big| \frac{1}{m}\sum_{e \in E} \Prob_{u}\big( e \in D_{t} \big) - 2u(1-u)f_d(t_{n}) \Big| \\
			& \leq \frac{1}{m} \sum_{e \in E_{\star}} \frac{4t_{n}}{\frac{1}{5}\log_{d} n}+ \frac{1}{m}|E_{\star}^{c}| \overset{\P}{\longrightarrow} 0,
		\end{split}
	\end{equation*}
	where the last convergence follows from Proposition~\ref{prop:geometry-rrg}(iii).
\end{proof}


\nc{c:bound_meeting_time}

\subsection{Expectation for large times}
\label{subsec:2}

In this section we show how the behaviour in Proposition \ref{prop:mean_linear_time} can be extended up to the linear time scale. Clearly, we need a different argument, since we can no longer assume that the dual random walks starting at the extremes of a locally tree-like edge do not exit their neighbourhood of size $\tfrac{1}{5}\log_d(n)$. In order to proceed, we adapt the arguments in \cite[Lemmas 17 and 20]{CFRmultiple} to our framework, which will constitute a crucial tool for our analysis. In what follows, we fix 
\begin{equation}\label{sigma}
	\sigma_n = \lceil \log\log n \rceil^{2}.
\end{equation}
The following lemma says that $\whp$ for every pair of starting vertices that are more than $\sigma_n$ apart the probability that the two random walks meet before time $\log^3 n$ is exponentially small in $\sigma_n$.

\begin{lemma}[Meeting time of distant random walks]
	\label{lemma:far_particles}
	Let $G_{d,n}(\omega)$ be a regular random graph. Then there exists a constant $\uc{c:bound_meeting_time}>0$ such that 
	\begin{equation*}
		\ind \left\{\max_{x,y\colon\, {\rm dist}(x,y)>\sigma_n} \Prob\big(\tau_{\rm meet}^{x,y} \leq \log^3 n\big) 
		\leq \ee^{-\uc{c:bound_meeting_time} \sigma_n}\right\} \overset{\P}{\longrightarrow} 1.
	\end{equation*}
\end{lemma}

\begin{proof}
	The proof comes in several steps.		
	
	\medskip\noindent
	{\bf 1.}
	Consider two independent random walks $(X_t)_{t \geq 0}$ and $(Y_t)_{t \geq 0}$ that start from $x$ and $y$, respectively. The number of steps each random walk performs before time $\log^3 n$ has distribution ${\rm Poisson}(\log^3n)$. In particular, Markov's inequality with exponential moments implies that, for $\lambda= \log\log n$,
	\begin{equation*}
		\begin{split}
			&\Prob\bigg(\begin{array}{c}
				X_t \text{ or } Y_t \text{ performs more than}
				\log^4 n \text{ steps before time } \log^3 n
			\end{array}\bigg)\\
			&\leq 2 \ee^{-\lambda \log^4 n} \exp\big\{(\ee^{\lambda}-1)\log^3 n\big\}
			= 2\exp\big\{(\ee^{\lambda}-1)\log^3 n - \lambda \log^4 n\big\}\\
			&= 2\exp\big\{(\log n - 1 -\log\log n \log n)\log^3 n\big\}
			\leq \ee^{-\uc{c:bound_meeting_time} \sigma_n},
		\end{split}
	\end{equation*}
	provided $n$ is large enough. Therefore, we may consider the case where $X$ and $Y$ perform discrete-time random walks in which at every step one of the random walks is selected uniformly at random to move. We then just need to verify that $\whp$
	\begin{equation*}
		\Prob^{\rm dt}\big(\tau_{\rm meet}^{x,y} \leq 2\log^4 n\big) \leq \ee^{-\uc{c:bound_meeting_time} \sigma_n},
	\end{equation*}
	uniformly over the pairs $x$ and $y$ that are at least $\sigma_n$ apart, where $\Prob^{\rm dt}$ denotes the law of the two discrete-time random walks.
	
	\medskip\noindent
	{\bf 2.} 
	Let $Z_t$ denote the graph distance of $X_t$ and $Y_t$ at time $t$. Note that $Z_0 \geq \sigma_n$, and $Z_{t+1}-Z_{t} \in \{-1,0,1\}$ for every $t \geq 0$ (the case $Z_{t+1}-Z_t=0$ corresponds to jumps like the one in Figure~\ref{fig:jump_same_distance}). This implies that the random walks cannot meet before time $\sigma_n$. Recall from Proposition~\ref{prop:geometry-rrg} that $\whp$
	\begin{equation*}
		\tx(\mathfrak{B}_{v}(2\sigma_n)) \leq 1, \qquad \text{for all } v \in V.
	\end{equation*}
	
	\begin{figure}
		\begin{tikzpicture}
			\begin{scope}[rotate=-30]
				\draw(-0.25,0)--(2.25,0);
				\draw[rotate=60](-0.25,0)--(2.25,0);
				\draw[shift={(2,0)}, rotate=120](-0.25,0)--(2.25,0);
				\draw[->, blue, thick, out=-30, in=-150, shift={(2,0)}, rotate=120](0.05,-0.05) to (1.95,-0.05);
				\fill[black] (0,0) circle (0.05);
				\fill[black] (2,0) circle (0.05);
				\node[below] at (0,0){$X$};
				\node[below] at (1.8,0){$Y$};
			\end{scope}
		\end{tikzpicture}
		\caption{In a general $d$-regular graph the potential existence of cycles of length $3$ allows the distance between the two random walks to remain constant after $1$ step. In the picture, if $Y$ moves along the blue arrow, then it remains at distance $1$ to $X$. }
		\label{fig:cycle}
		\label{fig:jump_same_distance}
	\end{figure}
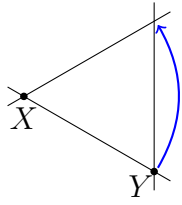
	
	\noindent		
	It is now enough to show that
	\begin{equation*}
		q_n = \Prob^{\rm dt} \Big(\tau_{\rm meet}^{x,y}\in \big\{\sigma_n,\dots,2\log^4 n\big\}\Big) \leq 4\log^8n \ee^{-2\uc{c:bound_meeting_time} \sigma_n} \leq \ee^{-\uc{c:bound_meeting_time}\sigma_n},
	\end{equation*}
	for some positive $\uc{c:bound_meeting_time}$ and all $n$ large enough. Put $\rho= \max\{t \leq \tau_{\rm meet}^{x,y}\colon Z_t \geq \sigma_n \}$. We rewrite the probability of interest as
	\begin{equation*}
		q_n \leq 2 \log^4 n \max_{r \leq 2\log^4 n}\Prob^{\rm dt} \big( \tau_{\rm meet}^{x,y} \leq 2\log^4 n, \rho=r\big).
	\end{equation*}
	We aim at providing bounds on the above probability that are uniform in $r \leq 2\log^4 n$. To that end, consider the ball of radius $2\sigma_n$ around $X_\rho$, which by assumption satisfies
	\begin{equation*}
		\tx(\mathfrak{B}_{X_{\rho}}(2\sigma_n)) \leq 1.
	\end{equation*}
	We distinguish between two cases: (1) the quantity above is $0$ (and thus $\mathfrak{B}_{X_{\rho}}(2\sigma_n)$ is a tree); (2) the tree excess of $\mathfrak{B}_{X_{\rho}}(2\sigma_n)$ is $1$.
	
	\medskip\noindent
	{\bf 3.}		
	First, consider the case when $\mathfrak{B}_{X_{\rho}}(2\sigma_n)$ is a tree. Call $\p$ the unique path joining $X_\rho$ and $Y_\rho$ within the tree, and recall that $|\p| = {\rm dist}(X_\rho,Y_\rho)=\sigma_n$. Let us argue that at least one of the random walks does $\sigma_n/2$ steps along $\p$. Indeed, suppose first that for all $t \in \{\rho, \dots, \tau_{\rm meet}^{x,y} \}$ neither $X_t$ nor $Y_t$ leaves the tree $\mathfrak{B}_{X_{\rho}}(2\sigma_n)$. In order for the two random walks to meet, they need to jointly traverse the path $\p$. Therefore, at least one of the two walks has to traverse half of the path $\mathfrak{p}$. If, on the other hand, one of the two random walks leaves $\mathfrak{B}_{X_{\rho}}(2\sigma_n)$, then the claim is trivial: indeed, the other walk must go through the path $\mathfrak{p}$ in order to have consistency with the event $\{\rho=r\}$. Here, note that when we say that the first/second random walk must have made $\sigma_n/2$ steps along $\p$, we do not require that these steps are made consecutively. Nonetheless, under $\{\rho=r\}$, requiring that at least one of the two random walks makes $\sigma_n/2$ steps along $\p$ within time $2 \log^4 n$ is equivalent to requiring that
	\begin{equation*}
		\hat{\tau}_{r} = \inf\{u \geq r\colon\, {\rm dist}(X_u,Y_r) \wedge {\rm dist}(X_r,Y_u) \leq \tfrac{\sigma_n}{2}\}
	\end{equation*}
	satisfies $\hat{\tau}_{r} \leq 2 \log^4 n$. (Note that because $\hat{\tau}_{r}$ is smaller than the exit time of $\cB_{X_{\rho}}(2\sigma_n)$, we can couple the two random walks with the random walks in $\mathcal{T}_d$.) We can now bound, uniformly in $r \leq 2 \log^4 n$,
	\begin{equation*}
		\Prob^{\rm dt}\big(\tau_{\rm meet}^{x,y}\le 2 \log^{4} n, \rho=r\big) 
		\leq \Prob^{\rm dt}\big(\hat{\tau}_r \leq 2 \log^4 n, \rho=r\big) 
		\leq 2 \Pr\big(\tau_{\rm abs} \leq 2 \log^4 n\big),
	\end{equation*}
	where $\tau_{\rm abs}$ is the absorption time at the origin of a biased random walk on $\Z$, starting at $\sigma_n/2$ and stepping to the right with probability $\frac{d-1}{d}$ and to the left with probability $\frac{1}{d}$. Hence, by the classical Gambler's ruin argument,
	\begin{equation}
		\label{eq:gambler_abs}
		\Pr\big(\tau_{\rm abs} \leq 2 \log^4 n\big) \leq \Pr\big(\tau_{\rm abs} < \infty\big)
		= \Big(\frac{1}{d-1}\Big)^{\sigma_n/2} \leq 2^{-\frac{\sigma_n}{2}}.
	\end{equation} 
	The claim follows by choosing $\uc{c:bound_meeting_time} \leq \frac{\log 2}{5}$.
	
	\begin{figure}[htbp]
		\begin{tikzpicture}
			\draw (-1,0) -- (0,3)--(1,0)--(-1,0);
			\fill[black] (-0.5,1.5) circle (0.05);
			\fill[black] (0,3) circle (0.05);
			\draw[thick, blue, dash pattern= on 3pt off 5pt] (-0.5,1.5) -- (0,3);
			\draw[thick, red, dash pattern= on 3pt off 5pt, dash phase=4pt] (-0.5,1.5) -- (0,3);
			\node[above] at (0,3){$X_{\rho}$};
			\node[left] at (-0.5,1.5){$Y_{\rho}$};
			
			\begin{scope}[shift={(3,0)}]
				\draw (-1,0) -- (0,3)--(1,0)--(-1,0);
				\fill[black] (-0.5,1.5) circle (0.05);
				\fill[black] (0,3) circle (0.05);
				\draw[thick, red] (-0.5,1.5) -- (0,3);
				\draw[thick, blue] (-0.5,1.5) -- (-0.3, 1.5) -- (0,3);
				\node[above] at (0,3){$X_{\rho}$};
				\node[left] at (-0.5,1.5){$Y_{\rho}$};
			\end{scope}
			
			\begin{scope}[shift={(6,0)}]
				\draw (-1,0) -- (0,3)--(1,0)--(-1,0);
				\fill[black] (-0.5,1.5) circle (0.05);
				\fill[black] (0,3) circle (0.05);
				\draw[thick, red] (-0.5,1.5) -- (0,3);
				\draw[thick, blue] (-0.5,1.5) -- (-0.25, 0.75) -- (0.05, 0.75) -- (0,3);
				\node[above] at (0,3){$X_{\rho}$};
				\node[left] at (-0.5,1.5){$Y_{\rho}$};
			\end{scope}
			
			\begin{scope}[shift={(9,0)}]
				\draw (-1,0) -- (0,3)--(1,0)--(-1,0);
				\fill[black] (-0.5,1.5) circle (0.05);
				\fill[black] (0,3) circle (0.05);
				\draw[thick, red] (-0.25,2.25) -- (0,3);
				\draw[thick, blue] (-0.25, 2.25) -- (0.05, 2.25)-- (0,3);
				\draw[thick, blue, dash pattern= on 3pt off 5pt] (-0.5,1.5)--(-0.25,2.25);
				\draw[thick, red, dash pattern= on 3pt off 5pt, dash phase=4pt] (-0.5,1.5)--(-0.25,2.25);
				\node[above] at (0,3){$X_{\rho}$};
				\node[left] at (-0.5,1.5){$Y_{\rho}$};
				\node[left] at (-0.25,2.25){$z_{\kappa}$};
			\end{scope}
			
		\end{tikzpicture}
		\vspace{0.2cm}
		\caption{A sketch of the $4$ possible ways the two paths $\p$ and $\tilde\p$ described below can look like. In the first picture, $\kappa=\sigma_n$. In the second and third picture, $\kappa=0$. In the last picture, $\kappa$ is a non trivial value in the interval $(0,\sigma_n)$.
		}
		\label{fig:tree-like_neighborhoods}
	\end{figure}
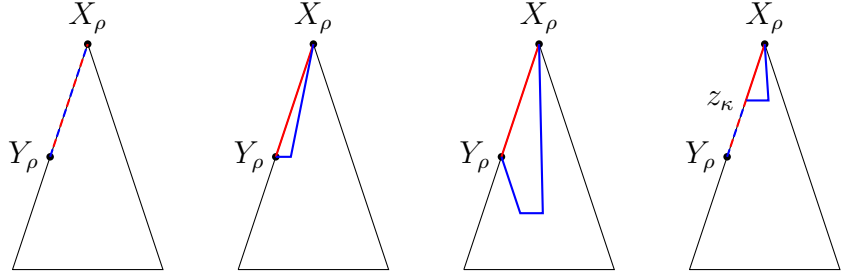
	
	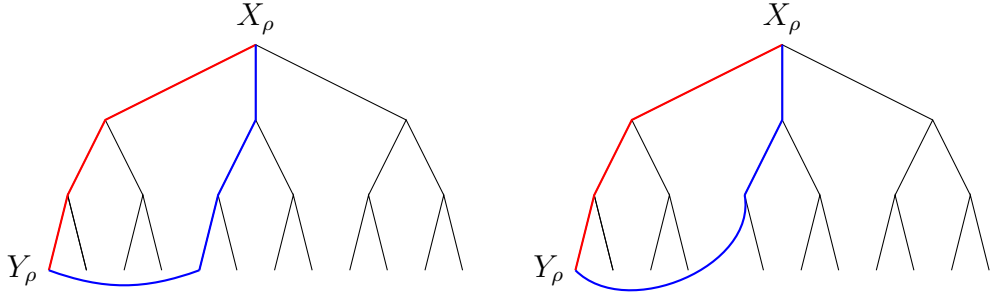
\begin{figure}[htbp]
		\begin{tikzpicture}
			\draw(1,2)--(3,3)--(3,2);
			\draw(3,3)--(5,2);
			\draw(0.5,1)--(1,2)--(1.5,1);
			\draw(2.5,1)--(3,2)--(3.5,1);
			\draw(4.5,1)--(5,2)--(5.5,1);
			\foreach \x in {0.5, 1.5, 2.5, 3.5, 4.5, 5.5}{
				\draw(\x-0.25,0)--(\x,1)--(\x+0.25,0);
			};
			\draw (0.5,1)--(0.75,0);
			\node[left] at (0.25,0){$Y_{\rho}$};
			\node[above] at (3,3){$X_{\rho}$};
			\draw[thick, red] (0.25,0)--(0.5,1) -- (1,2) -- (3,3);
			\draw[thick, blue] (2.25,0)--(2.5,1) -- (3,2) -- (3,3);
			\draw[thick, blue, out=-20, in=-160] (0.25,0) to (2.25,0);
			
			\begin{scope}[shift={(7,0)}]
				\draw(1,2)--(3,3)--(3,2);
				\draw(3,3)--(5,2);
				\draw(0.5,1)--(1,2)--(1.5,1);
				\draw(2.5,1)--(3,2)--(3.5,1);
				\draw(4.5,1)--(5,2)--(5.5,1);
				\foreach \x in {0.5, 1.5, 3.5, 4.5, 5.5}{
					\draw(\x-0.25,0)--(\x,1)--(\x+0.25,0);
				};
				\draw(2.5,1)--(2.75,0);
				\draw (0.5,1)--(0.75,0);
				\draw[thick, red] (0.25,0)--(0.5,1) -- (1,2) -- (3,3);
				\draw[thick, blue] (2.5,1) -- (3,2) -- (3,3);
				\draw[thick, blue, out=-45, in=-80] (0.25,0) to (2.5,1);
				\node[left] at (0.25,0){$Y_{\rho}$};
				\node[above] at (3,3){$X_{\rho}$};
			\end{scope}
			
		\end{tikzpicture}
		\caption{On the left $|\tilde\p|=|\p|+1$, on the right $|\tilde\p|=|\p|$. }
		\label{fig:how_it_goes_wrong}
	\end{figure}
	
	\medskip\noindent		
	{\bf 4.}
	Next, consider the case when $\tx( \mathfrak{B}_{X_{\rho}}(2\sigma_n) )=1$. Then there are at most two paths joining $Y_{\rho}$ to $X_{\rho}$ in $\mathfrak{B}_{X_{\rho}}(2\sigma_n)$ (see Figures~\ref{fig:tree-like_neighborhoods}--\ref{fig:how_it_goes_wrong}). Call these two paths $\p = (z_{0} = Y_{\rho}, z_{1}, \dots, z_{\sigma_n} = X_{\rho})$ and $\tilde{\p} = (\tilde{z}_{0} = Y_{\rho}, \tilde{z}_{1}, \dots, \tilde{z}_{\ell} = X_{\rho})$, with $\sigma_n \leq \ell \leq 3 \sigma_n+1$ (the case $\ell=\sigma_n$ is exemplified in Figure~\ref{fig:how_it_goes_wrong}). Let $\kappa$ denote the index of the last vertex such that the paths coincide, i.e.,
	\begin{equation*}
		\kappa = \sup \big\{i \leq \sigma_n\colon\, z_{i} = \tilde{z}_{i} \big\}.
	\end{equation*}
	We split the two paths as $\p = (\p_{\kappa}, \p')$ and $\tilde{\p} = (\p_{\kappa}, \p'')$, where $\p_{\kappa} = (z_{0}, z_{1}, \dots, z_{\kappa})$. If $\kappa \geq \frac{1}{2}\sigma_n$, then the argument from the previous paragraph can be easily adapted as follows. We first note that, in order for the random walks to meet, either $Y$ or $X$ must reach $z_{\frac{1}{2}\kappa}$. If $Y$ is the random walk that reaches this vertex, then the comparison with the biased random walk is done in the same way as before. If the random walk that reaches $z_{\frac{1}{2}\kappa}$ is $X$, then we still need to account for the time it takes $X$ to arrive at $z_{\kappa}$, which can be taken care of via a union bound, yielding an additional factor $2\log^4 n$ in  the calculation. More precisely, consider the random variables
	\begin{equation*}
		\tilde{\tau}_{r}^{1} = \inf\{u \geq r\colon\, X_{u}=z_{\frac{1}{2}\kappa}\}, \qquad 
		\tilde{\tau}_{r}^{2} = \inf\{u \geq r\colon\,Y_{u}=z_{\frac{1}{2}\kappa}\}.
	\end{equation*}
	We can bound
	\begin{equation*}
		\begin{split}
			\Prob^{\rm dt}\big(\tau_{\rm meet}^{x,y} \leq 2\log^4 n, \rho=r\big) 
			&\leq \Prob^{\rm dt} \Big(\tilde{\tau}_{r}^{1} \wedge \tilde{\tau}_{r}^{2} \leq 2\log^{4} n, \rho=r\Big) \\
			&\leq \Prob^{\rm dt} \Big(\tilde{\tau}_{r}^{1} \leq 2\log^4 n, \rho=r\Big) 
			+ \Prob^{\rm dt} \Big( \tilde{\tau}_{r}^{2} \leq 2\log^4 n, \rho=r\Big) \\
			&\leq \Pr\big(\tau_{\rm abs} \leq 2 \log^4 n\big) 
			+ \sum_{i=1}^{2\log^4 n} \Pr\big(\tau_{\rm abs} \leq \log^4 n-i\big) \\
			&\leq (\log^4 n+1) 2^{-\frac{\sigma_n}{4}}. 
		\end{split}
	\end{equation*}
	For the bound in the second line, the first term follows an analogous reasoning as in the tree case. To handle the second term, we first apply union bounds on the first time the random walk $X$ reaches $z_{\kappa}$, and the result follows from the strong Markov property and the same argument as in the tree case. Conversely, if $\kappa \leq \frac{1}{2}\sigma_n$, then both $\p'$ and $\p''$ are long paths (of size at least $\frac{1}{2}\sigma_n$). Once again, in order for the random walks to meet, at least one of them must traverse at least half of one of the paths $\p'$ or $\p''$ (because after the removal of $z_k$ the connected component of $X_{\rho}$ is a tree). The argument follows the same lines as in the previous case, and we refrain from spelling out the details. 
\end{proof}

\begin{proposition}[Exponential scaling of meeting time starting in stationarity]
	\label{lemma:firstvisitDelta}
	Let $G_{d,n}(\omega)$ be a regular random graph. There exists a sequence of random variables $\Theta_d=(\Theta_d^{(n)})_{n\in\N}$ such that
	\begin{equation}
		\label{eq:lemmaFVTLdelta_ct}
		\sup_{t \geq 0} \left|\frac{\Prob(\tau_{\rm meet}^{\pi\otimes\pi}>t)}{\exp\left(-\frac{2\:\Theta_d^{(n)} \: t}{n}\right)} -1 \right| 
		\overset{\P}{\longrightarrow}0\:,
	\end{equation}
	and
	\begin{equation*}
		\Theta_d^{(n)} \overset{\P}{\longrightarrow} \theta_d^{(n)} = \frac{d-2}{d-1}.
	\end{equation*}
\end{proposition}

\begin{proof}
	We show that for the discrete-time version of the process, in which at each time step one randomly chosen walk performs a jump, we have
	\begin{equation}
		\label{eq:lemmaFVTLdelta}
		\sup_{t \geq 0} \left|\frac{\Prob^{\rm dt}(\tau_{\rm meet}^{\pi\otimes\pi}>t)}{\left(1-\frac{\:\Theta_d}{n}\right)^t} -1\right| 
		\overset{\P}{\longrightarrow}0.
	\end{equation}
	The claim in \eqref{eq:lemmaFVTLdelta_ct} follows by Poissonisation. For the discrete-time process we can exploit the full power of the FVTL in Theorem \ref{t:fvtl}. The proof comes in several steps.
	
	\medskip\noindent
	{\bf 1.}
	Consider the multi-graph $\Gamma$ with vertex set $V_\Gamma= \{(x,y)\colon\, x,y\in[n],\: x\neq y\} \cup \{\Delta\}$, where $\Delta$ is the merge of the vertices on the diagonal that retains the edges, i.e., the edge set $E_\Gamma$ is such that
	\begin{equation*}
		((x,y),(v,w)) \in E_\Gamma \text{ if and only if } 
		\begin{cases}
			x=v \text{ and } (y,w)\in E,\\
			y=w \text{ and } (x,v)\in E,
		\end{cases}
	\end{equation*}
	and
	\begin{equation*}
		(\Delta ,(v,w)) \in E_\Gamma \text{ if and only if } (v,w)\in E,
	\end{equation*}
	with each edge adjacent to $\Delta$ having multiplicity $2$. Roughly, $\Gamma$ is such that every vertex except $\Delta$ has degree $2d$, while $\Delta$ has degree $4m=2dn$. It follows that the sum of the degrees in $\Gamma$ equals $2m_\Gamma=2d n(n-1)+ 2dn=2dn^2$, and that the number of vertices is given by $N=n^2-n+1$. Throughout the proof, $\Pr$ denotes the distribution of the simple random walk on $\Gamma$, and $P_\Gamma$ denotes the associated transition matrix. The stationary distribution $\pi_{\Gamma}$ of this random walk is given by 
	\begin{equation*}
		\pi_\Gamma(\Delta)=\frac{1}{n},
	\end{equation*}
	while, for all $(x,y)$ with $x\neq y$,
	\begin{equation*}
		\pi_\Gamma(x,y)=\frac{1}{n^2}.
	\end{equation*}
	To understand the reason underlying the definition of the multi-graph $\Gamma$, note that
	\begin{equation}\label{eq:identity}
		\mathbf{P}^{\rm dt}(\tau_{\rm meet}^{\pi\otimes\pi}>t)=\Pr(\tau_{\Delta}^{\pi_{\Gamma}}>t), \quad t \geq 0.
	\end{equation}
	\medskip\noindent
	{\bf 2.}	
	Concerning the mixing time in \eqref{eq:tmix}, it is not hard to show that we can choose $T_{\rm mix}\le \log^2 n$. For completeness, we provide a full proof by adapting the proof of \cite[Lemma 12]{CFRmultiple} to our asynchronous setting. As shown in \cite{friedman2008proof}, the second largest eigenvalue of $P$, denoted by $\lambda_\star$, is bounded away from $1$ $\whp$, from which it follows via \cite[Corollary 12.13]{LevPer:AMS2017} that the same holds for the matrix $P_{\otimes 2}=\frac{1}{2}(P\otimes I+I\otimes P)$. Write $\lambda_{\otimes 2}$ (respectively, $\lambda_\Gamma$) to denote the spectral gap of the matrix $P_{\otimes 2}$ (respectively, $P_\Gamma$). By \cite[Theorem 13.10]{LevPer:AMS2017} and with the abbreviation
	\begin{equation}
		\label{eq:conductance}
		\Phi_{\otimes 2} = \min_{ S \subset V^2 \::\: \pi_{\otimes 2}(S) \leq \frac{1}{2}} 
		\frac{\sum_{\x \in S}\pi_{\otimes 2}(\x)\sum_{\y \in S^c}P_{\otimes 2}(\x,\y)}{\pi_{\otimes 2}(S)}
	\end{equation}
	for the \emph{conductance} of the product chain, we know that
	\begin{equation}
		\label{eq:cheeger}
		\tfrac12\Phi_{\otimes 2}^2\le
		1-\lambda_{\otimes 2}\le 2\Phi_{\otimes 2}.
	\end{equation}
	Therefore $\Phi_{\otimes 2}$ is bounded away from $0$ \whp. By the construction of $\Gamma$, we must have $\Phi_{\otimes 2}\le \Phi_\Gamma$, where the latter is the conductance of the Markov chain $P_\Gamma$. Indeed, the graph $\Gamma$ is constructed from $G\times G$ by contracting some vertices and retaining the edges. Moreover, again by \eqref{eq:cheeger} to the Markov chain $P_\Gamma$,
	\begin{equation}
		{1-\lambda_{\Gamma}}\ge \tfrac12\Phi_{\Gamma}^2\ge \tfrac12\Phi_{\otimes 2}^2>0.
	\end{equation}
	Hence, we deduce that $\lambda_\Gamma$ must be bounded away from $1$ \whp. Finally, by \cite[Theorem 12.4]{LevPer:AMS2017}, we conclude that $O(\log n)$ steps suffice to get \eqref{eq:tmix}.
	
	\medskip\noindent
	{\bf 3.}
	Thus, Assumptions (1) and (3) in Theorem \ref{t:fvtl} are satisfied $\whp$ with $N=n^2-n+1$, $Q_N=P_\Gamma$, $\mu_N=\pi_\Gamma$. Moreover, choosing $x=\Delta$ we  easily check the $\whp$ existence of a unique quasi-stationary distribution, thanks to the fact that $d\ge 3$, $G$ is connected, and the two random walks move asynchronously. In order to use the FVTL, we are left to compute the expected number of returns to $\Delta$ before time $T_{\rm mix} \leq \log^2 n$, starting at $\Delta$. For $t \geq 1$, put
	\begin{equation*}
		R_t(\Delta) = \sum_{s=0}^t \Pr\big(X_s=\Delta \mid X_0= \Delta\big).
	\end{equation*}
	We will show that, for $T=\log^2 n$,
	\begin{equation}
		\label{RofDelta}
		R_{T}(\Delta) \overset{\P}{\longrightarrow} 1+ \frac{1-\theta_d}{\theta_d} =\frac{1}{ \theta_d }= \frac{d-1}{d-2}.
	\end{equation}
	The term $1$ comes from the fact that the random walk starts at $\Delta$.
	
	\medskip\noindent
	{\bf 4.}	
	We start by giving the heuristics of the proof. Call $(x,y)$ the vertex of $\Gamma$ visited at time $1$. By Proposition~\ref{prop:geometry-rrg}(ii), the probability that the edge $e=(x,y) \in E$ is ${\rm LTLE}(\tfrac{1}{5}\log_{d} n\big)$ is $1-o(1)$. On this event, the probability that there will be a second visit to $\Delta$ before $T_{\rm mix}$ is upper bounded by $\frac{1}{d-1}$. A simple combination of these arguments yields an upper bound. As for the lower bound, the probability that the random walk on $\Gamma$ starting at $(x,y)$ reaches a vertex $(v,w)$ with ${\rm dist}_G(v,w) \geq \sigma_n$ before reaching $\Delta$ is $\theta_d+o(1)$. Once at $(v,w)$, the probability that the random walk visits $\Delta$ before $T_{\rm mix}$ is $o(1)$. Hence, up to $o(1)$ corrections, $R_{T_{\rm mix}}(\Delta)$ can be estimated by the expectation of a geometric random variable with parameter $\theta_d+o(1)$, representing the first excursion of the random walk that hits $(v,w)$ with ${\rm dist}_G(v,w) \geq \sigma_n$ before returning to $\Delta$.
	
	\medskip\noindent
	{\bf 5.}		
	With this heuristics in mind, we are ready to provide the proof of \eqref{RofDelta}. We will do so by first assuming that
	\begin{equation}
		\label{eq:missing1}
		\Pr\big(\tfrac{1}{5} \log_d n \leq \tau_{\Delta}^+ \leq T \mid X_0 = \Delta\big) \overset{\P}{\longrightarrow} 0
	\end{equation}
	and
	\begin{equation}
		\label{eq:missing2}
		\Pr\big(\tau_{\Delta}^+ \leq \tfrac{1}{5}\log_d n \mid X_0 = \Delta\big) 
		\overset{\P}{\longrightarrow} 1-\theta_d.
	\end{equation}
	Write
	\begin{equation}
		\label{eq:equality_R}
		\begin{split}
			R_{T}(\Delta) 
			&= 1 + \sum_{s=1}^{T} \Pr\big(X_s = \Delta \mid X_0 = \Delta\big) \\
			&= 1 + \sum_{r=1}^{T} \sum_{s=r}^{T} \Pr\big(X_s = \Delta, \tau_{\Delta}^+ = r \mid X_0 = \Delta\big) \\
			&= 1 + \sum_{r=1}^{T} \sum_{s=r}^{T} \Pr\big(X_s = \Delta \mid X_r = \Delta\big)  
			\Pr\big(\tau_{\Delta}^+ = r \mid X_0 = \Delta \big) \\
			&= 1 + \sum_{r=1}^{T} \Pr\big(\tau_{\Delta}^+ = r \mid X_0 = \Delta\big) 
			\sum_{s=0}^{T-r} \Pr\big(X_s = \Delta \mid X_0 = \Delta\big) \\
			&= 1 + \sum_{r=1}^{T} \Pr\big(\tau_{\Delta}^+ =r \mid X_0 = \Delta \big) R_{T-r}(\Delta).
		\end{split}
	\end{equation}
	To establish the upper bound in \eqref{RofDelta}, observe that \eqref{eq:equality_R} yields
	\begin{equation*}
		R_{T}(\Delta) \leq 1 + \Pr\big(\tau_{\Delta}^+ \leq T \mid X_0 = \Delta\big) R_{T}(\Delta),
	\end{equation*}
	which yields
	\begin{equation*}
		R_{T}(\Delta) \leq \Big(1- \Pr\big(\tau_{\Delta}^+ \leq T \mid X_0=\Delta\big)\Big)^{-1},
	\end{equation*}
	and the claim follows by combining \eqref{eq:missing1} and \eqref{eq:missing2}.
	
	\medskip\noindent
	{\bf 6.}	
	To establish the lower bound in \eqref{RofDelta}, observe that \eqref{eq:equality_R} yields
	\begin{equation*}
		\begin{split}
			R_{T}(\Delta) 
			& \geq 1 + \sum_{r=1}^{\tfrac{1}{5} \log_d n } \Pr\big(\tau_{\Delta}^+ = r \mid X_0 = \Delta\big) R_{T-r}(\Delta) \\
			& \geq 1 + \Pr\big(\tau_{\Delta}^+ \leq \tfrac{1}{5} \log_d n \mid X_0 = \Delta \big) R_{T - \tfrac{1}{5} \log_d n}(\Delta).
		\end{split}
	\end{equation*}
	By iterating the estimate above, we obtain
	\begin{equation}
		\label{eq:RofDeltaLB}
		R_{T}(\Delta) \geq \sum_{i=0}^{5T/\log_d n} \Pr\big(\tau_{\Delta}^+ \leq \tfrac{1}{5} \log_d n \mid X_0
		=\Delta \big)^i = \frac{1}{\theta_d} - o_\P(1),
	\end{equation}
	where the last equality follows from \eqref{eq:missing2} and uses the choice $T= \log^2 n$ made earlier. This concludes the verification of the hypotheses of Theorem \ref{t:fvtl}. 
	
	\medskip\noindent
	{\bf 7.}
	To conclude the proof, we still need to verify \eqref{eq:missing1} and \eqref{eq:missing2}. We start with~\eqref{eq:missing1}. By Proposition~\ref{prop:geometry-rrg}(iii), we have
	\begin{equation}
		\label{eq:good_first_jump}
		\Pr\big(X_1 \in \text{LTLE}\big( \tfrac{1}{5}\log_d n\big) \mid X_0 = \Delta\big) \overset{\P}{\longrightarrow}1.
	\end{equation}
	Indeed, by construction, the random walk on $\Gamma$ starting at $\Delta$ after a single step visits a $(x,y)\in E$ uniformly at random. Assume that $X_{1} = (x,y) \in {\rm LTLE} \big(\tfrac{1}{5} \log_d n\big)$, and let $\tilde{\tau}$ denote the first time at which the process on $\Gamma$ is found at some $(w,z)$ with ${\rm dist}_G(w,z)=\sigma_n$.  Then, up to time $\tau_{\Delta} \wedge \tilde{\tau}$, the random walk $(X_t)_{t \geq 1}$ can be coupled with a biased random walk on $\Z$ (given by the distance between the two vertices) that starts from $1$, jumps to the right (respectively, left) with probability $\tfrac{d-1}{d}$ (respectively, $\tfrac{1}{d}$). Call this random walk $(W_t)_{t\ge0}$, and let $\tau_j$ , $j\in\Z$, be its first hitting of $j$. Then a gambler's ruin argument yields
	\begin{equation*}
		\Exp_1\big( \tau_0 \wedge \tau_{\sigma_n}\big) \leq \frac{d\sigma_n}{d-2},
	\end{equation*}
	which implies
	\begin{equation}
		\label{eq:bound_probability}
		\Pr\big(\tau_0 \wedge \tau_{\sigma_n} > \tfrac{1}{5} \log_d n -1\mid X_0=1 \big) 
		\leq \Pr \big(\tau_0 \wedge \tau_{\sigma_n} > \sigma_n^3\mid X_0=1\big) \leq \frac{1}{\sigma_n}
	\end{equation}
	for all $n$ sufficiently large. Compute
	\begin{equation*}
		\begin{split}
			&\Pr\big(\tau_{\Delta}^+ \wedge \tilde{\tau} \leq \tfrac{1}{5} \log_d n \mid X_0 = \Delta\big) \\
			&\:\:\:\geq \sum_{(x,y) \in (V_\Gamma \setminus \{\Delta\}) \cap {\rm LTLE}\big(\tfrac{1}{5} \log_d n\big)} 
			\Pr\big(X_1 = (x,y)\mid X_0=\Delta\big) \Pr\big(\tau_{\Delta}^+\wedge \tilde{\tau} \leq \tfrac{1}{5} \log_d n - 1\mid X_0=(x,y)\big) \\
			& \overset{\eqref{eq:bound_probability}}{\geq} \Big(1-\frac{1}{\sigma_n}\Big) 
			\Pr\big(X_1 \in \text{LTLE}\big(\tfrac{1}{5} \log_d n\big) \mid X_0 = \Delta \big).
		\end{split} 
	\end{equation*}
	Note that the last line above converges to 1 in probability as $n \to \infty$ (see~\eqref{eq:good_first_jump}), which with the help of Lemma~\ref{lemma:far_particles} concludes the verification of \eqref{eq:missing1}.
	
	\medskip\noindent
	{\bf 8.}
	We are left with proving \eqref{eq:missing2}. In this case, we can assume once again that $X_1 \in \text{LTLE}(\tfrac{1}{5} \log_d n)$, which in particular implies that, up to time $\tfrac{1}{5} \log_d n$, the random walk can be perfectly coupled with a biased random walk on $\Z$ that starts at $1$ and jumps to the right (respectively, left) with probability $\tfrac{d-1}{d}$ (respectively, $\tfrac{1}{d}$). The claim follows by noting that
	\begin{equation*}
		\begin{split}
			\Pr\big(\tau_{\Delta}^+ \leq \tfrac{1}{5} \log_d n \mid X_0 = \Delta\big) 
			& = o_\P(1) + \Pr\big(\tau_0 \leq \tfrac{1}{5} \log_d n - 1\mid X_0=1\big) \\
			& = o_\P(1) + \Pr\big(\tau_0 < \infty\mid X_0=1\big) = o(1)+\frac{1}{d-1},
		\end{split}
	\end{equation*}
	which concludes the verification of \eqref{eq:missing2}.
\end{proof}

The next lemma identifies the tail of the meeting time of random walks starting from afar.

\begin{lemma}[Tail of meeting time of distant random walks]
	\label{lemma:exponential_meeting_far_particles}
	Let $G_{d,n}(\omega)$ be a regular random graph and consider two independent random walks. Then
	\begin{equation}
		\label{eq:lemmaFVTL_continuous}
		\max_{x,y\colon\, {\rm dist}(x,y) \geq \sigma_n} \sup_{t \geq 0} \Big|\Prob\big(\tau^{x,y}_{\rm meet} > t\big) 
		- \ee^{-\frac{2t\theta_{d}}{n}}\Big| \overset{\P}{\longrightarrow}0.
	\end{equation}
\end{lemma}

\begin{proof}
	We need to deduce from Proposition \ref{lemma:firstvisitDelta} an (additive) analogous result when the two random walks start far away. For that we use Lemma~\ref{lemma:far_particles}.
	
	For simplicity we work again in discrete time, so that what we actually need to show is
	\begin{equation}
		\label{eq:lemmaFVTL}
		\max_{x,y\colon\, {\rm dist} (x,y) \geq \sigma_n} \max_{t \geq 0} \left|\Prob^{\rm dt}(\tau_{\rm meet}^{x,y}>t) 
		- \Big(1-\frac{\theta_d}{n}\Big)^t \right| \overset{\P}{\longrightarrow} 0.
	\end{equation}
	Fix $x$ and $y$ such that ${\rm dist}(x,y) \geq \sigma_n$, and let $\cE$ be the event that the two random walks meet before time $T$. Since $T = \log^{2} n$, Lemma~\ref{lemma:far_particles} states that $\Pr(\cE) \leq \ee^{-\uc{c:bound_meeting_time}\sigma_n}$ $\whp$ for every pair $(x,y)$ as above. Hence, if $t \leq T$, then
	\begin{equation}
		\label{eq:meeting_time_estimate}
		\min_{x,y\colon\,{\rm dist}(x,y) \geq \sigma_n} \Prob^{\rm dt}\big(\tau_{\rm meet}^{x,y} > t\big) 
		\geq 1 - \ee^{-\uc{c:bound_meeting_time}\sigma_n}.
	\end{equation}
	On the other hand, since $T\ge T_{\rm mix}$, if $t>T$, then
	\begin{equation}
		\label{eq:mixingtime}
		\max_{(x,y)\in V_\Gamma } |\Pr\big(X_t^{x}=v,Y^{y}_t=w\big) - \pi_\Gamma(v,w)| \leq \frac{1}{n^6}. 
	\end{equation}
	Therefore
	\begin{equation*}
		\begin{split}
			\Prob^{\rm dt} \big(\tau_{\rm meet}^{x,y} > t\big) 
			& = \sum_{(v,w)\in V_\Gamma \setminus \{\Delta\}} \Pr\big(\tau_\Delta^{(x,y)}>T, \, 
			X^x_{T}=v, \, Y^y_{T} = w\big) \Pr\big( \tau^{(v,w)}_\Delta > t-T \big) \\
			& \leq \sum_{(v,w) \in V_\Gamma \setminus \{\Delta\}} \Pr\big(X^x_{T} = v,\, Y^y_{T} = w\big) 
			\Pr \big(\tau^{(v,w)}_\Delta > t-T\big) \\
			&\leq \sum_{(v,w) \in V_\Gamma \setminus \{\Delta\}} \pi_\Gamma(v,w) \Pr\big(\tau^{(v,w)}_\Delta>t-{T}\big) 
			+ \frac{1}{n^{6}} \\
			& \leq \Pr\big(\tau^{\pi_\Gamma}_\Delta>t-T\big) + \frac{1}{n^{4}}.
		\end{split}
	\end{equation*}
	On the other hand,
	\begin{equation*}
		\begin{split}
			&\Prob^{\rm dt} \big( \tau_{\rm meet}^{x,y} > t \big)\\ 
			& = \sum_{(v,w) \in V_\Gamma\setminus\Delta} \Pr\big(\tau_\Delta^{(x,y)} > T, \, X^x_{T} = v, \, 
			Y^y_{T} = w\big) \Pr\big(\tau^{(v,w)}_\Delta > t-T\big)\\
			& \geq \left(\sum_{(v,w) \in V_\Gamma \setminus \{\Delta\}} \Pr\big(X^x_{T}=v, \, Y^y_{T}=w\big) \Pr\big(\tau^{(v,w)}_\Delta > t-T \big) \right)
			- \Prob^{\rm dt}\big(\tau_{\rm meet}^{x,y} \leq T\big) \\
			&\overset{\eqref{eq:mixingtime}}{\geq} \sum_{(v,w) \in V_\Gamma \setminus \{\Delta\} } \pi_{\Gamma}(v,w) \Pr\big( \tau^{(v,w)}_\Delta > t-T \big) - \frac{1}{n^{6}} - \Prob^{\rm dt}\big( \tau_{\rm meet}^{x,y} \leq T \big) \\
			& \overset{\eqref{eq:meeting_time_estimate}}{=} 
			\Pr\big( \tau^{\pi_{\Gamma}}_\Delta > t-T \big)-\frac{1}{n^{4}} - \ee^{-\uc{c:bound_meeting_time}\sigma_n}.
		\end{split}
	\end{equation*}
	This, when combined with Proposition~\ref{lemma:firstvisitDelta}, implies the claim. Indeed, thanks to \eqref{eq:lemmaFVTLdelta} and \eqref{eq:identity} we have
	\begin{equation*}
		\Pr\big( \tau^{\pi_{\Gamma}}_\Delta > t-T \big)=(1+o(1)) \Big(1-\frac{\theta_d}{n} \Big)^{t-T}.
	\end{equation*}
	Moreover,
	\begin{equation*}
		\Big| \Big(1-\frac{\theta_d}{n} \Big)^t-\Big(1-\frac{\theta_d}{n} \Big)^{t-T} \Big| 
		= \Big(1-\frac{\theta_d}{n} \Big)^t \Big| 1- \Big(1-\frac{\theta_d}{n} \Big)^{-T} 
		\Big| \leq \ee^{-\frac{t\theta_{d}}{n}}\Big|1- \Big(1-\frac{\theta_d}{n} \Big)^{-T} \Big|
	\end{equation*}
	and the latter converges to zero as $n \to \infty$, uniformly in $t \geq 0$.
\end{proof}

The next proposition identifies the tail of the meeting time of random walks starting from the vertices of a typical edge.

\begin{proposition}[Tail of meeting time of adjacent random walks on a LTLE]
	\label{lemma:meeting_near_particles}
	Let $G_{d,n}(\omega)$ be a regular random graph. Then, for all $t_n$ such that $\lim_{n\to\infty} t_n = \infty$ and $\lim_{n\to\infty} t_n/n = s \in [0,\infty)$, 
	\begin{equation*}
		\max_{e=(x,y)\in {\rm LTLE}( \frac{1}{5} \log_d n)} \Big| \Prob\big(\tau_{\rm meet}^{x,y} > t_{n}) - \theta_d\ee^{-2s\theta_d} \Big| \overset{\P}{\longrightarrow}0.
	\end{equation*}
\end{proposition}

\begin{proof}
	Note that it is enough to assume that $t_{n} \geq \sigma_n^3$, since otherwise the claim follows from Lemma \ref{lemma:general_graph}. Let $\tau_{\rm far}^{x,y}$ be the first time at which the two random walks are at distance at least $\sigma_n$, and $\tau_{\rm exit}^{x,y}$ the first time at which one of the two random walks exits the locally tree-like neighbourhood of radius $\tfrac{1}{5} \log_d n$ of $e$. Put $\tau_0^{x,y} = \tau_{\rm meet}^{x,y}\wedge \tau_{\rm far}^{x,y}\wedge \tau_{\rm exit}^{x,y}$. For simplicity, we start by analysing the discrete-time process. Up to time $\tau_0^{x,y}$, the process can be coupled to a biased random walk on $\Z$ that steps to the right with probability $\frac{d-1}{d}$ and to the left with probability $\frac{1}{d}$. In particular, 
	\begin{equation*}
		\Expect^{\rm dt}\big[\tau_0^{x,y}\big] \leq \frac{d\sigma_n}{d-2}
	\end{equation*}
	because $\tau_0^{x,y}$ is bounded by the exit time of the biased random walk from the interval $[1,\sigma_n]$, and therefore
	\begin{equation}
		\label{eq:tau-exit}
		\begin{split}
			\Prob^{\rm dt}\big(\tau^{x,y}_0 = \tau^{x,y}_{\rm exit}\big) 
			&\leq \Prob^{\rm dt}\big(\tau^{x,y}_{\rm exit} \leq \sigma_n^3 \big) 
			+ \Prob^{\rm dt}\big(\tau^{x,y}_0 > \sigma_n^3\big) \\
			&\leq 0 + \frac{1}{\sigma_n^3}\frac{d\sigma_n}{d-2}=o(1),
		\end{split}
	\end{equation}
	where the last term follows from the Markov inequality. Similarly, by a gambler ruin argument,
	\begin{equation*}
		\Prob^{\rm dt}\left(\tau^{x,y}_{\rm far}>\sigma_n^3 \right)=o(1),
	\end{equation*}
	and therefore
	\begin{equation}
		\label{eq:first}
		\begin{aligned}
			\Prob\big(\tau_{\rm meet}^{x,y} > t_{n}\big)
			&=\Prob\big(\tau_{\rm meet}^{x,y} > t_{n}, \, \tau_0^{x,y} = \tau_{\rm far}^{x,y}\le\sigma_n^3\big) + o(1).
		\end{aligned}
	\end{equation}
	Recall the usual asuymptotic notation in which, for any two sequences $f_n$ and $g_n$, $f_n\sim g_n$ is shorthand for $f_n=(1+o(1))g_n$.
	The probability in the right-hand side of \eqref{eq:first} can be written, in the discrete-time setting, as 
	\begin{equation*}
		\begin{aligned}
			&\Prob^{\rm dt}\big(\tau_{\rm meet}^{x,y} > t_{n}, \, \tau_0^{x,y} = \tau_{\rm far}^{x,y}\le \sigma_n^3\big) \\ 
			&=\sum_{u=0}^{\sigma_n^3}\sum_{(v,w)\colon\,{\rm dist}(v,w)>\sigma_n}\Prob^{\rm dt}(\tau_0^{x,y}=\tau_{\rm far}^{x,y}=u,\:
			(X_{\tau_0^{x,y}},Y_{\tau_0^{x,y}})=(v,w))	\Prob(\tau_{\rm meet}^{v,w}>t_n-u)\\
			&\sim\sum_{u=0}^{\sigma_n^3}\sum_{(v,w)\colon\,{\rm dist}(v,w)>\sigma_n} \Prob^{\rm dt}(\tau_0^{x,y}=\tau_{\rm far}^{x,y}=u,\: 
			(X_{\tau_0^{x,y}},Y_{\tau_0^{x,y}})=(v,w))	\left(1-\frac{\theta_d}{n}\right)^{t_{n}-u}\\
			&=\left(1-\frac{\theta_d}{n}\right)^{t_{n}}\sum_{u=0}^{\sigma_n^3} \Prob^{\rm dt}(\tau_0^{x,y}=\tau_{\rm far}^{x,y}=u)\left(1-\frac{\theta_d}{n}\right)^{-u}	\\
			&\sim\left(1-\frac{\theta_d}{n}\right)^{t_{n}}\sum_{u=0}^{\sigma_n^3} \Prob^{\rm dt}(\tau_0^{x,y}=\tau_{\rm far}^{x,y}=u)
			\sim \ee^{-s\theta_d} \Prob^{\rm dt}(\tau_0^{x,y}=\tau_{\rm far}^{x,y}\le\sigma_n^3)
			\sim \theta_d\,\ee^{-s\theta_d} \qquad \whp,
		\end{aligned}
	\end{equation*}
	where in the first approximation we use Lemma~\ref{lemma:exponential_meeting_far_particles} and in the last approximation that the probability for the biased random walk to exit $[1,\sigma_n]$ at the right converges to $\theta_d$ as $n\to\infty$. In order to pass to the continuous-time setting, it is enough to realise that, for all $\varepsilon\in(0,s)$,
	\begin{equation*}
		\begin{aligned}
			\Prob\big(\tau_{\rm meet}^{x,y} > t_{n}\big)
			&= \sum_{u=0}^\infty \Pr\left({\rm Poisson}(2t_n)=u\right) \Prob^{\rm dt} \big( \tau_{\rm meet}^{x,y} > u \big)\\
			&= \sum_{u=2(s-\varepsilon)n\vee 0}^{2(s+\varepsilon)n} \Pr\left({\rm Poisson}(2sn)=u\right)	
			\Prob^{\rm dt}\big( \tau_{\rm meet}^{x,y} > u\big)+o_\P(1)\\
			&= \theta_d\,\ee^{-2s\theta_d}+O_\P(\varepsilon),
		\end{aligned}
	\end{equation*}
	after which the claim follows by taking the limit $\varepsilon \downarrow 0$.
\end{proof}

The next two statements turn the result in Proposition \ref{lemma:meeting_near_particles} into a statement about the expected density of discordant edges at time $O(n)$ for the voter model.

\begin{corollary}[Discordant edges at small distances]
	\label{coro:largetimes}
	Let $G_{d,n}(\omega)$ be a $d$-regular random graph, let $e=(x,y) $ be in ${\rm LTLE}(\tfrac{1}{5} \log_d n)$, and consider the voter model with initial density $u \in (0,1)$. Then, for all $t_n$ such that $\lim_{n\to\infty} t_n=\infty$ and $\lim_{n\to\infty} t_n/n = s \in [0,\infty)$,
	\begin{equation*}
		\Big| \Prob_u\big(e \in D_{t_{n}} \big) - 2u(1-u)\,\theta_d\,\ee^{-2s\theta_d} \Big|\overset{\P}{\longrightarrow}0.
	\end{equation*}
\end{corollary}

\begin{proof}
	The claim follows from Proposition \ref{lemma:meeting_near_particles} via duality.
\end{proof}

\begin{proposition}[Expected density of discordant edges]
	\label{prop:linearmean}
	Let $G_{d,n}(\omega)$ be a $d$-regular random graph, and consider the voter model with initial density $u \in (0,1)$. Then, for all $t_n$ such that $\lim_{n\to\infty} t_n=\infty$ and $\lim_{n\to\infty} t_n/n = s \in [0,\infty)$,
	\begin{equation*}
		\Big| \Expect_u[\cD_{t_n}] - 2u(1-u)\,\theta_d\,\ee^{-2s\theta_d}\Big| \overset{\P}{\longrightarrow}0.
	\end{equation*}
\end{proposition}

\begin{proof}
	The claim follows from  Corollary \ref{coro:largetimes} and the same argument as in the proof of Proposition \ref{prop:mean_linear_time}.
\end{proof}

\begin{proof}[Proof of Theorem~\ref{theorem:expectation}]
	Proposition~\ref{prop:mean_linear_time} takes care of the case where $t_n$ is bounded, Proposition~\ref{prop:linearmean} takes care of the case where $t_n \to \infty$ and $t_n/n \to s \in [0,\infty)$, while the case $t_n/n \to \infty$ follows from domination.
\end{proof}


\section{Beyond the expectation}
\label{sec:beyondexp}

In this section we prove Theorem \ref{thm:beyondthemean}. We start with a proposition that implies the first statement of the theorem.

\begin{proposition}[Concentration for worst case initialization]
	\label{prop:pointwise-concentration}
	Consider times $t_n$ such that $t_n\to\infty$ and $t_n/n\to 0$. Then, for any $\varepsilon>0$,
	\begin{equation*}
		\sup_{\xi \in \{0,1\}^{V}} \Prob_\xi\left(\left|\cD_{t_n}-\Expect_\xi\left[\cD_{t_n}\right]\right|>\varepsilon\right) \overset{\P}{\longrightarrow} 0.
	\end{equation*}
\end{proposition}

\begin{proof}
	Putting $E_\star={\rm LTLE}(\tfrac{1}{5} \log_d n)\subset E$ and recalling that $|E_\star|=m-o(n)$ \emph{whp} by Proposition \ref{prop:geometry-rrg} (ii), we have
	\begin{equation}
		\label{eq:2ndmom}
		\begin{aligned}
			\Expect[|D_t|^2]
			&= \sum_{e,e'\in E} \Prob\left(e,e'\in D_t \right)=\sum_{e,e'\in E_\star} \Prob\left(e,e'\in D_t \right)+o(n^2) \\
			&= \sum_{e\in E_\star}\left[\sum_{e'\in E_\star,\:{\rm dist(e,e')>\sigma_n}}\Prob\left(e,e'\in D_t \right)+O(d^{\sigma_n})\right]+o(n^2),
		\end{aligned}
	\end{equation}
	with $\sigma_n$ as in~\eqref{sigma}.
	Fix $e=(x,y),e'=(x',y')\in E_\star$ such that ${\rm dist}(e,e')>\sigma_n$, and consider the event 
	\begin{equation*}
		\cE=\cE_{e,e'}\coloneqq\left\{\tau_{\rm meet}^{x,x'}\wedge 
		\tau_{\rm meet}^{x,y'}\wedge\tau_{\rm meet}^{y,x'}\wedge\tau_{\rm meet}^{y,y'}> t\right\}.
	\end{equation*}
	We observe that, on the event $\cE$, the events $\{e\in D_t \}$ and $\{e'\in D_t\}$ are negatively correlated. Indeed, denote by $\sigma_e$ (respectively, $\sigma_{e'}$) a realisation of length $t$ of the two independent random walk trajectories starting at $x$ and $y$ (respectively, $x'$ and $y'$). Let $\cH(\sigma_{e})$ denote the set of possible realisations of $\sigma_{e'}$ that never meet the trajectory $\sigma_e$. We then have
	\begin{equation}\label{negCorrPair}
		\begin{split}
			\Prob_\xi(e,e'\in D_t,\:\cE)
			&=\sum_{\sigma_e\in \{e\in D_t \}}\sum_{\sigma_{e'}\in \{e'\in D_t \}\cap\cH(\sigma_e)} 
			\Prob_\xi(\sigma_e)\Prob_\xi\left(\sigma_{e'}\mid \sigma_e \right)\\
			&=\sum_{\sigma_e\in \{e\in D_t \}}\sum_{\sigma_{e'}\in \{e'\in D_t \}\cap\cH(\sigma_e)} \Prob_\xi(\sigma_e)\Prob_\xi\left(\sigma_{e'} \right)\\
			&\le\sum_{\sigma_e\in \{e\in D_t \}}\sum_{\sigma_{e'}\in \{e'\in D_t \}} \Prob_\xi(\sigma_e)\Prob_\xi\left(\sigma_{e'} \right)
			= \Prob_\xi(e\in D_t)\:\Prob_\xi(e'\in D_t),
		\end{split}
	\end{equation} 
	which gives the claimed negative dependence.
	
	Furthermore, by Lemma~\ref{lemma:exponential_meeting_far_particles}, we have $\Prob(\cE^c)=o_\P(1)$, which together with~\eqref{negCorrPair} guarantees that
	\begin{equation*}
		\Prob_\xi \left( e,e'\in D_t \right) \leq \Prob_\xi\left(e,e'\in D_t ,\cE \ \right)+\Prob_\xi(\cE^c)\le \Prob_\xi(e\in D_t)\Prob_\xi\left(e'\in D_t\right)+o_\P(1).
	\end{equation*}
	Inserting the latter inequality into \eqref{eq:2ndmom}, we get
	\begin{equation*}
		\begin{aligned}
			\Expect_\xi[|D_t|^2]&= \sum_{e\in E_\star}\left[\sum_{e'\in E_\star,\:{\rm dist}(e,e')>\sigma_n}\Prob_\xi\left(e,e'\in D_t \right)+O(d^{\sigma_n})\right]+o(n^2)\\
			&\le \sum_{e\in E_\star}\left[\sum_{e'\in E_\star,\:{\rm dist}(e,e')>\sigma_n}\left[\Prob_\xi(e\in D_t)\Prob_\xi\left( e'\in D_t\right)+o_\P(1)\right]+O(d^{\sigma_n})\right]+o(n^2)\\
			&\le \Expect_\xi[|D_t|]^2+o_\P(n^2)+o(n^2)
			= (1+o_\P(1))\Expect_\xi[|D_t|]^2.
		\end{aligned}
	\end{equation*}
	Therefore, by Chebyshev's inequality,
	\begin{equation*}
		\Prob_\xi\left(\big| |D_t|-\Expect_\xi|D_t| \big|>\varepsilon m \right) 
		\leq \Prob_\xi\left(\big| |D_t|-\Expect_\xi|D_t| \big|>\varepsilon \Expect_\xi|D_t| \right)
		\leq \frac{\var_\xi(|D_t|)}{\varepsilon^2\Expect_\xi[|D_t|]^2}=o_\P(1).
	\end{equation*}
\end{proof}

We next consider times $t_n$ such that $t_n/n \to s \in (0,\infty)$ and use results from \cite{CCC}. Let
\begin{equation}\label{eq:gamma-n}
	\gamma_{n} = \Expect[\tau_{\textnormal{meet}}^{\pi \otimes \pi}]
\end{equation}
be the expected meeting time of two independent continuous-time simple random walks with uniformly chosen initial positions. \cite{CCC} (see Theorems 2.1 and 2.2 therein) provide conditions on the underlying sequence of graphs under which the process $(\cB^{n}_{\gamma_{n} t})_{t \geq 0}$, recording the proportion of type-1 vertices (recall \eqref{eq:def-frac-of-blue}), converges as $n \to \infty$ in the space of c\`{a}dl\`ag paths to the Fisher-Wright diffusion $(\bar\cB_t)_{t \geq 0}$ for $\gamma_n$ as in \eqref{eq:gamma-n}. In the case of $d$-regular random graphs,
\begin{equation}
	\label{eq:asymptotics_gamma}
	\gamma_{n} = \frac{1}{2}\theta_{d}^{-1}\,n + o_\P(1),
\end{equation}
and the requirements in \cite{CCC} are easily verified \emph{whp}.

The following lemma is the central estimate that we need in order to establish convergence of the proportion of discordant edges.

\begin{lemma}[Linking density of discordant edges to density of type-1]
	\label{lemma:l1_estimate}
	Let $s_{n}$ converge to a finite value. Then
	\begin{equation*}
		\Expect_{u} \Big[ \Big|\cD_{s_{n} \gamma_{n}} 
		- 2\theta_{d}\cB^{n}_{s_{n} \gamma_{n}}(1-\cB^{n}_{s_{n} \gamma_{n}}) \Big| \Big] 
		\overset{\P}{\longrightarrow} 0.
	\end{equation*}
\end{lemma}

Before proving this lemma, we conclude the proof of Theorem~\ref{thm:beyondthemean}.

\begin{proof}[Proof of Theorem~\ref{thm:beyondthemean} (ii)]
	Recall that we are assuming that $t_n/n \to s$. Consider $s_n=t_n/\gamma_n$ and note that 
	$\lim_{n\to\infty}s_n = 2s\,\theta_{d}$,
	due to~\eqref{eq:asymptotics_gamma}. Hence Lemma~\ref{lemma:l1_estimate} yields
	\begin{equation}
		\label{eq:l1_estimate_proof}
		\Expect_{u} \Big[ \Big|\cD_{t_{n}} - 2\theta_{d}\cB^{n}_{t_{n}}(1-\cB^{n}_{t_{n}}) \Big| \Big]  
		\overset{\P}{\longrightarrow} 0.
	\end{equation}
	Moreover, since $x \mapsto x(1-x)$ is continuous and so is the Fisher-Wright diffusion $(\bar\cB_{t})_{t\geq 0}$,
	Proposition~\ref{prop:convergence} implies that
	\begin{equation*}
		2\theta_{d}\cB^{n}_{t_{n}}(1-\cB^{n}_{t_{n}}) \quad \text{ converges in distribution to } \quad 
		2\theta_{d}\bar\cB_{2s\,\theta_{d}}(1-\bar\cB_{2s\,\theta_{d}}).
	\end{equation*}
	In view of~\eqref{eq:l1_estimate_proof}, the same holds \emph{whp} for $\cD_{t_{n}}$, which concludes the proof.
\end{proof}

To prove Lemma~\ref{lemma:l1_estimate}, we use the following concentration estimate derived in \cite{CCC}. 
\begin{lemma}[Concentration on arbitrary connected graphs]
	\label{lemma:CCC}
	Fix a finite connected graph $G$ on $n$ vertices and $m$ edges and let $(\eta_{t})_{t \geq 0}$ denote the voter model evolving on $G$. For any $0 \leq s \leq t$,
	\begin{equation*}
		\begin{split}
			\sup_{\eta_{0}} \Big| \Expect_{\eta_{0}}[\cD_{t}]-2\Prob [\tau^{e}>s]\frac{1}{n^2}\sum_{x}\eta_{0}(x) & \sum_{x}(1-\eta_{0}(x)) \Big| \\
			& \leq \Prob \big[ \tau^{e} \in (s, t] \big] + 4\Prob \big[ \tau^{e}>s \big] d_{\rm TV}(t-s),
		\end{split}
	\end{equation*}
	where $\tau^{e}$ denotes the meeting time of two random walks starting from the two vertices at the end of a uniformly chosen edge of $G$, and $d_{TV}$ is defined as in~\eqref{TV}.
\end{lemma}

\noindent
The proof of Lemma \ref{lemma:CCC} is provided in~\cite[Lemma 6.1]{CCC}. Note that it applies to any deterministic connected graph on which the voter dynamics takes place. 

We will also need the asymptotic estimate
\begin{equation}
	\label{eq:meeting_log^2}
	\Prob \big[\tau^{e} \geq k\log^{2} n \big]  \overset{\P}{\longrightarrow} \theta_{d}, \qquad k \in \N,
\end{equation}
which follows from Proposition \ref{lemma:meeting_near_particles}.

\begin{proof}[Proof of Lemma~\ref{lemma:l1_estimate}]
	Fix $\delta_{n} = 2\log^{2} n$, and split
	\begin{equation}
		\label{eq:estimate_1}
		\begin{aligned}
			&\Expect_{u} \Big[ \Big|\cD_{s_{n} \gamma_{n}} - 2\theta_{d}\cB_{s_{n} \gamma_{n}}
			(1-\cB_{s_{n} \gamma_{n}}) \Big| \Big] 
			\leq \Expect_{u} \Big[ \Big|\cD_{s_{n} \gamma_{n}} - 2\theta_{d}\cB_{s_{n} 
				\gamma_{n}-\delta_{n}}(1-\cB_{s_{n} \gamma_{n}-\delta_{n}}) \Big| \Big] \\
			& \qquad + 2\theta_{d} \Expect_{u} \Big[ \Big|\cB_{s_{n} \gamma_{n}-\delta_{n}}
			(1-\cB_{s_{n} \gamma_{n}-\delta_{n}})
			- \cB_{s_{n} \gamma_{n}}(1-\cB_{s_{n} \gamma_{n}}) \Big| \Big].
		\end{aligned}
	\end{equation}
	We will show that each of the expectations in the right-hand side above converges to zero. 
	
	\medskip\noindent
	{\bf 1.}
	Apply the Markov property at time $s_n\gamma_n-\delta_n$, to obtain
	\begin{equation}
		\label{eq:estimate_2}
		\begin{aligned}
			&\Expect_{u} \Big[ \Big|\cD_{s_{n} \gamma_{n}} - 
			2\theta_{d}\cB_{s_{n} \gamma_{n}-\delta_{n}}(1-\cB_{s_{n} \gamma_{n}-\delta_{n}}) \Big| \Big] 
			\leq \sup_{\eta_{0}}\Expect_{\eta_{0}} \Big[ \Big|\cD_{\delta_{n}} - 2\theta_{d}\cB_{0}
			(1-\cB_{0}) \Big| \Big] \\
			& \leq \sup_{\eta_{0}} \Expect_{\eta_{0}} \Big[ \big|\cD_{\delta_{n}} - \Expect_{\eta_{0}}[\cD_{\delta_{n}}] \big| \Big] 
			+ \sup_{\eta_{0}} \Big|\Expect_{\eta_{0}}[\cD_{\delta_{n}}] - 2\theta_{d}\frac{1}{n^{2}}\sum_{x}\eta_{0}(x) 
			\sum_{x}(1-\eta_{0}(x)) \Big|.
		\end{aligned}
	\end{equation}
	The first expectation in the right-hand side of~\eqref{eq:estimate_2} converges to zero by Proposition~\ref{prop:pointwise-concentration} and the fact that $|\cD_{\delta_{n}} - \Expect_{\eta_{0}}[\cD_{\delta_{n}}]|$ is bounded by one. To control the second expectation in \eqref{eq:estimate_2}, apply Lemma~\ref{lemma:CCC} with $t=\delta_{n}= 2\log^{2} n$ and $s=\log^{2} n$, to obtain
	\begin{equation*}
		\begin{split}
			& \sup_{\eta_{0}} \Big|\Expect_{\eta_{0}}[\cD_{\delta_{n}}] - 2\theta_{d}\frac{1}{n^{2}}\sum_{x}\eta_{0}(x) \sum_{x}(1-\eta_{0}(x)) \Big| \\ 
			& \leq \sup_{\eta_{0}} \Big|\Expect_{\eta_{0}}[\cD_{\delta_{n}}] - 2\Prob \big[\tau^{e} \geq \log^{2} n \big]\frac{1}{n^{2}}
			\sum_{x}\eta_{0}(x) \sum_{x}(1-\eta_{0}(x)) \Big| + 2\Big|\Prob \big[\tau^{e} \geq \log^{2} n \big] - \theta_{d} \Big| \\
			& \leq  \Prob \big[ \tau^{e} \in (\log^{2} n, 2\log^{2} n] \big] + 4\Prob \big[ \tau^{e}>\log^{2} n \big] 
			\textnormal{d}_{TV}(\log^{2} n)+ 2\Big|\Prob \big[\tau^{e} \geq \log^{2} n \big] - \theta_{d} \Big|.
		\end{split}
	\end{equation*}
	Note that both
	\begin{equation*}
		\Prob \big[ \tau^{e} \in (\log^{2} n, 2\log^{2} n] \big] = \Prob \big[ \tau^{e} > 2\log^{2} n \big] - \Prob \big[ \tau^{e} > \log^{2} n \big]
	\end{equation*}
	and 
	\begin{equation*}
		\Big|\Prob \big[\tau^{e} \geq \log^{2} n \big] - \theta_{d} \Big|
	\end{equation*}
	converge to zero \emph{whp} by~\eqref{eq:meeting_log^2}, while $d_{\rm TV}(\log^{2} n)$ goes to zero thanks to Proposition \ref{prop:geometry-rrg} (iv). Hence also the second expectation in~\eqref{eq:estimate_2} vanishes, and so we have
	\begin{equation}
		\label{eq:estimate_3}
		\Expect_{u} \Big[ \Big|\cD_{t_{n} \gamma_{n}} - 2\theta_{d}\cB_{t_{n} 
			\gamma_{n}-\delta_{n}}(1-\cB_{t_{n} \gamma_{n}-\delta_{n}}) \Big| \Big]  \overset{\P}{\longrightarrow}  0.
	\end{equation}
	
	\medskip\noindent
	{\bf 2.} It remains to show that the second term in the right-hand side of Eq.~\eqref{eq:estimate_1} vanishes \emph{whp}, for which we argue by continuity in the proper topology. For an arbitrary $T>0$, let $[0,T] \times \mathcal{D}[0,T]$ denote the product of $[0,T]$ with the c\`adl\`ag space $\mathcal{D}[0,T]$ endowed with the $J_1$-Skorohod topology. For a given $s>0$ and an arbitrary evaluation function $h\colon\, [0,T] \times \mathcal{D}[0,T] \to \R$ defined via $h(t, \phi) = \phi(t)$, we consider the incremental function $\tilde{h}\colon\, [0,T]^{2} \times \mathcal{D}[0,T] \to \R$ given by
	\begin{equation*}
		\tilde{h}(t, u, \phi) = |h(t, \phi)-h(u, \phi)|,
	\end{equation*}
	and note that, by continuity of the modulus function and  Lemma~\ref{lemma:continuity}, all points in $[0,T]^{2} \times \mathcal{C}[0,T]$ are continuity points of $\tilde{h}$. In particular, since $\frac{\delta_{n}}{\gamma_{n}} \to 0$ and with $s=\lim_{n\to\infty}s_{n}<T$,  we have that  
	
	\begin{equation}
		\begin{aligned}
			&\Expect_{u} \Big[ \Big|\cB_{s_{n} \gamma_{n}}
			(1-\cB_{s_{n} \gamma_{n}}) -\cB_{s_{n} \gamma_{n}-\delta_{n}}
			(1-\cB_{s_{n} \gamma_{n}-\delta_{n}}) \Big| \Big]
			= \Expect_{u} \Big[ \Big|h\big( s_{n}, (\cB_{u \gamma_{n}}(1-\cB_{u \gamma_{n}}))_{u \in [0,T]} \big)\\
			&- h\big( s_{n}-\frac{\delta_{n}}{\gamma_{n}}, (\cB_{u \gamma_{n}}
			(1-\cB_{u \gamma_{n}}))_{u \in [0,T]} \big)\Big| \Big]  
			= \Expect_{u} \Big[ \tilde{h}\big( s_{n}, s_n-\frac{\delta_{n}}{\gamma_{n}} ,(\cB_{u \gamma_{n}}
			(1-\cB_{u \gamma_{n}}))_{u \in [0,T]} \big)\Big]\\
			& \label{eq:estimate_4}
			\overset{\P}{\longrightarrow}  \Expect_{u} \Big[ \tilde{h}\big( s ,s,(\bar\cB_{u }(1-\bar\cB_u))_{u\in [0,T]} \big)\Big]  =0,
		\end{aligned}
	\end{equation}
	where the limit is justified by~\eqref{eq:dB} and Proposition~\ref{prop:convergence} .
	
	\medskip\noindent
	{\bf 3.}
	Combine~\eqref{eq:estimate_1},~\eqref{eq:estimate_3}, and~\eqref{eq:estimate_4}, to conclude the proof.
\end{proof}


\section{Uniform concentration}
\label{sec:concentration}

In this section we prove Theorem \ref{th:unif-concentration}, i.e., we sharpen the result in Theorem \ref{thm:beyondthemean}.1 to a uniform bound over sublinear times up to the scale $n^{1-o(1)}$. Note that, up to a union bound on the set of jump times for the process $\cD_t^n$, the proof amounts to showing the following.

\begin{proposition}[Strengthening of the pointwise concentration]
	\label{prop:unif-concentration}
	Consider the voter model with initial density $u\in(0,1)$ on a regular random graph $G_{d,n}(\omega)$. For any fixed $\varepsilon,\delta,a>0$, 
	\begin{equation*}
		\ind\left\{ \max_{t\le n^{1-\varepsilon}} \Prob_u\left(\left|\cD_t- \Expect_u[\cD_t]\right|>\delta\right)  \le n^{-a} \right\} 
		\overset{\P}{\longrightarrow} 1.
	\end{equation*}
\end{proposition}

\begin{proof}	
	The proof comes in several steps. To ease the reading, we drop the subindex $u$ from the notation of the voter measure $\Prob_{u}$ and abbreviate $p_t=\Expect[\cD_t]$.
	
	\medskip\noindent
	{\bf 1.}
	Fix
	\begin{equation*}
		t\le  n^{1-\varepsilon},\qquad  K_n=\log^2 n,
	\end{equation*}
	and let $A$ be a collection of $K_{n}$ edges of the graph sampled uniformly at random, independently of the voter dynamics. Call
	\begin{equation*}
		\cD^A_t\coloneqq\frac{1}{K_n}\sum_{e\in A}\ind_{\{e\in D_t\}}
	\end{equation*}
	the fraction of edges in the random set $A$ that are discordant at time $t$. By the triangle inequality, we have
	\begin{equation}
		\label{aaa}
		\Prob\left(\left|\cD_t- p_t\right|>\delta\right) \le \Prob\left(\left|\cD_t - \cD_t^A\right|>\tfrac12\delta\right) 
		+ \Prob\left(\left| \cD_t^A- p_t\right|>\tfrac12\delta\right).
	\end{equation}
	Note that, given $\cD_t=q$, the quantity $\cD_t^A$ is distributed as $Z/K_n$, where $Z\sim {\rm Bin}(K_n,q)$. Therefore, conditioning on $\cD_t=q\in[0,1]$, we can bound the first term on the right-hand side of \eqref{aaa} by means of the Chernoff bound
	\begin{equation}
		\label{bbb}
		\Prob\left(\left|\cD_t - \cD_t^A\right|>\tfrac12\delta\:\bigg\rvert\: \frac{|D_t|}{m}=q\right)
		\le 2\exp\left(-\frac{\delta^2}{12 q}K_n \right)\leq 2\exp\left(-\frac{\delta^2}{12}K_n\right).
	\end{equation}
	Since the bound in \eqref{bbb} is uniform over $q\in[0,1]$ and $t\le n^{1-\varepsilon}$, we conclude that, for every $a>0$, $\P$-a.s.,
	\begin{equation}\label{eq:Dt-A}
		\max_{t\le n^{1-\varepsilon}}\Prob\left(\left|\cD_t- \cD_t^A\right|>\tfrac12\delta\right) 
		\le  2 \exp\left(-\frac{\delta^2}{12} K_n\right)\le n^{-a},
	\end{equation}
	for all $n$ sufficiently large.
	
	\medskip\noindent
	{\bf 2.}
	We next show that a similar bound holds for the second term on the right-hand side of \eqref{aaa}. The latter will be proved by means of duality. Consider a system of $n$ independent random walks starting from the different vertices of $G$ and evolving independently. We consider a (multi)sub-set of $2K_n$ random walks starting at the extremes of the edges in the random set $A$. Note that these are distributed at time zero as $\otimes_{K_n}\nu$, where $\nu$ is the probability distribution on $[n]^2$ defined as
	\begin{equation*}
		\nu(x,y)=\pi(x)\frac{1}{d}\ind_{x\sim y},\qquad (x,y)\in[n]^2.
	\end{equation*}
	Moreover, since $\pi\equiv \frac{1}{n}$, the two marginal distributions of $\nu$ coincide with the stationary distribution, i.e., for every $x\in[n]$,
	\begin{equation*}
		\sum_{z\in[n]} \nu(x,z) = \sum_{z\in[n]}\nu(z,x)=\frac{1}{n}.
	\end{equation*}
	Observe that
	\begin{equation}
		\label{lll}
		p_t=\Expect[\cD_t]=\frac{1}{m}\sum_{e\in E}\Prob(e\in D_t)=\Prob_\nu(e\in D_t).
	\end{equation}
	In order to simplify the notation, when considering a system of $2K$ independent random walks starting at $\otimes_{K_n}\nu$, for each of the $K_n$ random edges $A=\{e_1,\dots, e_{K_n} \}$, we label the extremes as $e_j^-$ and $e_j^+$, for all $j\in\{1,\dots,K_n \}$. For $i,j\le K_n$, define the quantities
	\begin{equation}\label{eq:tef}
		\tau^{e_i,e_j}\coloneqq\tau_{\rm meet}^{e_i^-,e_j^-}\wedge \tau_{\rm meet}^{e_i^+,e_j^-}\wedge 
		\tau_{\rm meet}^{e_i^-,e_j^+}\wedge \tau_{\rm meet}^{e_i^+,e_j^+}.
	\end{equation}
	When $\tau^{e_i,e_j}\le t$ we say that \emph{the edges $i$ and $j$ interact before $t$}. We say that $\tau^{e_i,e_j}=0$ when $\{e_i^-,e_i^{+} \}\cap \{e_j^-,e_j^+ \}\neq\emptyset$. Note that, for all $i,j\le K_n$ and $t\ge0$,
	\begin{equation}
		\label{ccc}
		\Prob_{\otimes_{K_n}\nu}(\tau^{e_i,e_j}\le t) = \Prob_{\nu\otimes\nu}(\tau^{e_1,e_2}\le t) 
		\le 4 \Prob_{\pi\otimes \pi}(\tau_{\rm meet}^{x,y} \le t).
	\end{equation}
	Indeed, the inequality in \eqref{ccc} follows from the definition of $\nu$ and a union bound.
	
	We next argue that for all $t=o(n)$ the probability on the right-hand side of \eqref{ccc} can be bounded above $\whp$ by
	\begin{equation}
		\label{eq:ccc2}
		\Prob_{\pi\otimes \pi}(\tau_{\rm meet}^{x,y} \leq t) = O\left(\frac{t}{n}\right).
	\end{equation}
	Let us start by pointing out that the FVTL is not enough to deduce \eqref{eq:ccc2}. Indeed, the FVTL is particularly suited to have sharp estimates of the right tail of the hitting time of a set on time scales that are \emph{at least} as big as the mean hitting time. For this reason, we will exploit a result in \cite{AB1} that in a sense complements the FVTL on short time scales (see also \cite[Lemma 2.11]{hermon2021mean}).
	
	\begin{lemma}
		For an irreducible and reversible continuous-time Markov chain with stationary distribution $\mu$ and any subset of states $A$,
		\begin{equation}
			\label{eq:ccc3}
			\left|\Prob(\tau^\mu_A\le t)-\frac{t}{\Expect[\tau^\mu_A]} \right|\le \left(\frac{t}{\Expect[\tau^\mu_A]} \right)^2 
			+ \frac{1}{\lambda_\star\:\Expect[\tau^\mu_A]},\qquad0\le  t\le \Expect[\tau^\mu_A],
		\end{equation}
		where $\tau_A^\mu$ is the hitting time of the set $A$ when starting with distribution $\mu$, and $\lambda_{\star}$ denotes the minimal non-trivial eigenvalue of the infinitesimal generator of the process.
	\end{lemma}
	In our setting we can think of the Markov chain as the simple random walk on $G\otimes G$ and of $A$ as the diagonal, so that $\mu=\pi\otimes\pi$, $\lambda_{\star}\in(0,1)$ and $\Expect[\tau^\mu_A]=\Expect[\tau^{\pi\otimes\pi}_\Delta]=\Theta(n)$, where the latter follows by Proposition \ref{lemma:firstvisitDelta}. Note that, as soon as $t/n\to 0$, we have that the term on the right-hand side of \eqref{eq:ccc3} is $O(t/n)$, as for $t/\Expect[\tau^\mu_A]$, so that \eqref{eq:ccc2} immediately follows. Finally, as already pointed out in the second step of the proof of Proposition \ref{lemma:firstvisitDelta}, the fact that $\lambda_{\star}$ is bounded away from $0$ and $1$ as $n\to\infty$ has been proved by \cite{friedman2008proof}.
	
	The following lemma states that it is unlikely to have within $A$ a density of edges interacting before time $t$.
	
	\begin{lemma}
		\label{lemma:concentration_crw}
		Consider a system of $n$ independent random walks, each starting from a different vertex of a regular random graph $G_{d,n}(\omega)$. Select a sub(multi)set $A$ of $E$ of size $2K_n$ with joint distribution $\otimes_{K_n}\nu$. Fix $\varepsilon>0$ and $\gamma\in (0, 1)$, $t\le n^{1-\varepsilon}$ and $K_n=\log^2 n$. 
		Let $\mathcal{I}\subset A$ be the maximal subset such that
		\begin{equation*}
			\forall e\in \mathcal{I},\, \exists f\in A\setminus\{ e\}\text{ s.t. }\,\tau^{e,f}\le t\,,
		\end{equation*}
		and call
		\begin{equation*}
			\cE_{t,\gamma}:= \left\{|\mathcal{I}|>\gamma K_n\right\},
		\end{equation*}
		where $\tau^{e,f}$ is defined as in \eqref{eq:tef}. Then, for all $a>0$, 
		\begin{equation*}
			\Prob_{\otimes_{K_n} \nu}\left(\cE_{t,\gamma}\right) \le n^{-a} \quad \whp.
		\end{equation*}
	\end{lemma}
	
	Before proving Lemma \ref{lemma:concentration_crw}, we show how the latter can be exploited to conclude the proof of Proposition \ref{prop:unif-concentration}.
	
	\medskip\noindent
	{\bf 3.}	
	To conclude the proof of Proposition \ref{prop:unif-concentration}, we argue as follows. Let 
	\begin{equation}\label{H}
		\cH_t=\{| \cD^A_t-p_t|> \delta\}
	\end{equation}
	be the event of interest. Consider a system of $2K_n$ independent random walks starting at the extremes of the vertices in $A$, and let these evolve up to time $t$. Recall that $\mathcal{I} \subset A$ denotes the set of edges in $A$ interacting within time $t$ with other edges in $A$. Consider the event $\cE_{t, \gamma}$ in Lemma \ref{lemma:concentration_crw} for which we know that, for every constants $a,\gamma>0$, \whp
	\begin{equation*}
		\Prob(\cE_{t,\gamma})\le n^{-a}.
	\end{equation*}
	Therefore
	\begin{equation}\label{eq:Ht}
		\Prob\left(\cH_t\right) \le \Prob(\cH_t\cap \cE^c_{t,\gamma})+n^{-a}.
	\end{equation}
	In order to conclude the proposition, it suffices to bound the probability on the right-hand side of the above display.
	
	\medskip\noindent
	{\bf 4.}
	In order to get the desired bound we next show negative correlation for the discordancy of a collection of edges, in the spirit of the proof of Proposition \ref{prop:pointwise-concentration}.
	
	\begin{lemma}[Negative correlation under no-meeting events]
		\label{lemma:nega-dep}
		Consider a sub(multi)set of edges $A_k=\{e_1,\dots,e_k \}$ sampled according to $\otimes_k \nu$, and let $\cG_{t,k}$ be the event in which the $2k$ independent random walks starting at the extremes of the $k$ edges never meet before time $t$. Then, for all $j\in\{0,1,2,\dots,k\}$,
		\begin{equation*}
			\Prob_{\otimes_k\nu}\left(\text{Exactly $j$ of the $k$ edges are discordant at time $t$},\:\cG_{t,k} \right) \leq \Pr\left({\rm Bin}(k,p_t)=j \right).
		\end{equation*}
	\end{lemma}
	
	\begin{proof}
		For $e=(x,y)\in A_k$, call $\sigma_e$ a realisation of length $t$ of the two independent trajectories starting at $x$ and $y$. For $2\le j\le k$, call $\cH(\sigma_1,\dots,\sigma_{j-1})$ the set of the possible realisations of $\sigma_j$ that never meet the $2(j-1)$ trajectories $\sigma_1,\dots,\sigma_{j-1}$. Observe that
		\begin{equation*}
			\label{iii}
			\begin{split}
				&\Prob_{\otimes_k\nu}(e_i\in D_t,\forall i\le j,\:\:e_i\not\in D_t,\forall i> j,\:\cG_{t,k})\\
				&=\sum_{\sigma_1\in \{e_1\in D_t \}}\sum_{\sigma_2\in \{e_2\in D_t \}\cap\cH(\sigma_1)}
				\cdots\sum_{\sigma_k\in\{e_k\not\in D_t\}\cap \cH(\sigma_1,\dots,\sigma_{k-1})}
				\prod_{j=1}^k \Prob\left(\sigma_i\mid \cap_{j<i} \sigma_j \right)\\
				&=\sum_{\sigma_1\in \{e_1\in D_t \}}\sum_{\sigma_2\in \{e_2\in D_t \}\cap\cH(\sigma_1)}
				\cdots\sum_{\sigma_k\in\{e_k\not\in D_t\}\cap \cH(\sigma_1,\dots,\sigma_{k-1})}
				\prod_{j=1}^k \Prob\left(\sigma_i\right)\\
				&\le\sum_{\sigma_1\in \{e_1\in D_t \}}\sum_{\sigma_2\in \{e_2\in D_t \}}\cdots\sum_{\sigma_k\in\{e_k\not\in D_t\}}
				\prod_{j=1}^k \Prob\left(\sigma_i\right)\\
				&=\prod_{i=1}^j \Prob_{\otimes_k\nu}(e_i\in D_t)\prod_{i=j+1}^k \Prob_{\otimes_k\nu}(e_i\not\in D_t)\\
				&= p_t^j\left(1-p_t \right)^{k-j},
			\end{split}
		\end{equation*} 
		where in the last equality we use the identity in \eqref{lll}. By the product form of $\otimes_k\nu$ the result immediately follows.
	\end{proof}
	
	\medskip\noindent
	{\bf 5.}
	We next show how to exploit Lemma \ref{lemma:nega-dep} to conclude that 
	\begin{equation}
		\label{eq:finalgoal}
		\mathbf{P}(\cH_t\cap \cE_{t,\gamma}^c)\le n^{-a}, \quad \text{for all } a \geq 0,
	\end{equation}
	which implies, thanks to  \eqref{aaa}, \eqref{eq:Dt-A} and \eqref{eq:Ht}, the validity of Proposition \ref{prop:unif-concentration}.
	
	First of all, we split the event $\cH_t\cap \cE_{t,\gamma}^c$ into two parts, depending on the sign of the deviation:
	\begin{equation}
		\label{eq:split-dev}
		\mathbf{P}(\cH_t\cap \cE_{t,\gamma}^c)=\mathbf{P}(\cD_t^A-p_t>\delta,\,\cE_{t,\gamma}^c)+\mathbf{P}(p_t-\cD_t^A>\delta,\,\cE_{t,\gamma}^c).
	\end{equation}
	For the first probability in the right-hand side of \eqref{eq:split-dev} we estimate
	\begin{equation}
		\label{eq:est1}
		\begin{split}
			\mathbf{P}(\cD_t^A-p_t>\delta,\,\cE_{t,\gamma}^c)&=\mathbf{P}\left(\frac{1}{K_n}\sum_{e \in \mathcal{I}}
			\ind_{e\in D_t}+\frac{1}{K_n}\sum_{e \in A\setminus\mathcal{I}}\ind_{e\in D_t}-p_t>\delta,\,\cE_{t,\gamma}^c \right)\\
			&\le\mathbf{P}\left(\frac{1}{K_n}\sum_{e \in A\setminus\mathcal{I}}\ind_{e\in D_t}-p_t>\delta-\gamma,\,\cE_{t,\gamma}^c \right)\\
			&\le \mathbf{P}\left(\frac{1}{K_n-|\mathcal{I}|}\sum_{e \in A\setminus\mathcal{I}}\ind_{e\in D_t}>p_t+\delta-\gamma,\,\cE_{t,\gamma}^c \right)\\
			&\le \mathbf{P}\left({\rm Bin}(K_n-|\mathcal{I}|,p_t)>\bigg(1+\frac{\delta-\gamma}{p_t}\bigg)p_t(K_n-|\mathcal{I}|) \,,\cE_{t,\gamma}^c\right)\\
			&\le \ee^{-\Theta(\log^2 n)}.
		\end{split}
	\end{equation}
	Here, in the first inequality we use that  $|\mathcal{I}|/K_n\le \gamma$ under $\cE_{t,\gamma}^c$, in the third inequality we use Lemma \ref{lemma:nega-dep}, while the last inequality follows from the Chernoff bound by choosing $\gamma=\delta/2$. By the same argument, for the second quantity in the right-hand side of \eqref{eq:split-dev} we estimate
	\begin{equation}
		\label{eq:est2}
		\begin{split}
			\mathbf{P}(p_t-\cD_t^A>\delta,\,\cE_{t,\gamma}^c)&\le \mathbf{P}\left(p_t-\frac{1}{K_n}
			\sum_{e \in A\setminus\mathcal{I}}\ind_{e\in D_t}>\delta,\,\cE_{t,\gamma}^c \right)\\
			&\le \mathbf{P}\left(\frac{1-\gamma}{K_n-|\mathcal{I}|}\sum_{e \in A\setminus\mathcal{I}}
			\ind_{e\in D_t}<p_t-\delta,\,\cE_{t,\gamma}^c \right)\\
			&= \mathbf{P}\left(\sum_{e \in A\setminus\mathcal{I}}\ind_{e\in D_t}<\bigg(1-\frac{\delta-p_t\gamma}{p_{t}(1-\gamma)}\bigg)
			p_t(K_n-|\mathcal{I}|),\,\cE_{t,\gamma}^c \right)\\
			&\le  \mathbf{P}\left({\rm Bin}(K_n-|\mathcal{I}|,p_t)<\bigg(1-\frac{\delta-p_t\gamma}{p_{t}(1-\gamma)}\bigg)
			p_t(K_n-|\mathcal{I}|),\,\cE_{t,\gamma}^c \right)\\
			&\le \ee^{-\Theta(\log^2n)}.
		\end{split}
	\end{equation}
	Hence \eqref{eq:finalgoal} follows via \eqref{eq:est1} and \eqref{eq:est2}.
\end{proof}

\begin{proof}[Proof of Lemma \ref{lemma:concentration_crw}]
	The proof comes in several steps. For simplicity, we suppress the dependence on $\gamma$ from the notation.
	
	\medskip\noindent
	{\bf 1.} For a vertex $x\in[n]$, we call $\cP_x$ the subset of times $[0,t]$ in which the random walk starting at $x$ moves. Put
	\begin{equation*}
		\cP^{\rm tot}=\left\{\forall x\in[ n], \:\:|\cP_{x}|\le n^{\varepsilon/2} t \right\} 
	\end{equation*}
	and note that
	\begin{equation*}
		\Prob\left(\cP^{\rm tot}\right) \ge 1- n\Pr\left({\rm Poisson}(t)>n^{\varepsilon/2} t \right) \ge 1- n^{-a},
		\qquad \text{for all } a>0,
	\end{equation*}
	where the last inequality holds for all $n$ large enough. Therefore it is enough to show that, for all $a>0$,
	\begin{equation*}
		\Prob_{\otimes_{K_n} \nu}\left(\cE_{t,\gamma}\cap\cP^{\rm tot}\right) \le n^{-a} \qquad whp.
	\end{equation*}
	Due to the maximality of the set $\mathcal{I}$ we have
	\begin{equation}
		\label{eq:first-bound}
		\begin{aligned}
			\Prob_{\otimes_{K_n} \nu}\left( \cE_{t,\gamma}\cap\cP^{\rm tot}\right)
			&\leq \sum_{ I\subset A\colon\, | I|\ge \gamma K_n}\Prob_{\otimes_{K_n} \nu}\left(\mathcal{I}= I,\:\cP^{\rm tot}\right)\\
			&\le \sum_{C=\gamma K_n}^{K_n}\binom{K_n}{C}\Prob_{\otimes_{C} \nu}\left(\max_{e\in  I}\min_{f\in  I\setminus\{e \}}\tau^{e,f}
			\leq t,\:\cP^{\rm tot}\right),
		\end{aligned}
	\end{equation}
	where in the second line, since edges in $A$ are i.i.d.\ with law $\nu$, and with a slight abuse of notation, we can take for $I$ the set of the first $C$ edges of $A$. Given $I$, associate to it one possible collection of first interactions between its vertices. Call this set 
	\begin{equation*}
		B_I=\{(1,j_1),\dots,(C,j_C)\},
	\end{equation*}
	where $(\ell,j_\ell)$ stands for the event that the first edge with which the edge $e_{\ell}$ interacts is $e_{j_\ell}$. 
	
	\medskip\noindent
	{\bf 2.}	
	For every choice of $B_I$, we denote by $\cE_t(I,B_I)$ the event $\cE_{t,\gamma}$ in which the set of interacting edges is $I$ and $B_I=\{(1,j_1),\dots,(C,j_C) \}$ is the set of first interactions. Then
	\begin{equation}
		\label{fff}
		\Prob_{\otimes_C\nu}\left(\cE_t(I,B_I)\cap\cP^{\rm tot}\right) \leq \Prob_{\otimes_C \nu}\left(\tau^{e_{1},e_{j_1}}\le t,\cP^{\rm tot}(1)\right)
		\prod_{\ell=2}^C\Prob_{\otimes_C \nu}\left(\tau^{e_{\ell},e_{j_\ell}}\le t,\:\cP^{\rm tot}(\ell)\mid \cF_{\ell-1}\right),
	\end{equation}
	where, for $\ell\in\{1,\dots, C\}$, we put
	\begin{equation*}
		\cP^{\rm tot}(\ell):=\{|\cP_{e_{\ell}^-}|\le n^{\varepsilon/2} t\}
		\cap\{|\cP_{e_{\ell}^+}|\le n^{\varepsilon/2} t\}
		\cap\{|\cP_{e_{j_\ell}^-}|\le n^{\varepsilon/2} t\}
		\cap\{|\cP_{e_{j_\ell}^+}|\le n^{\varepsilon/2} t\}
	\end{equation*}
	and
	\begin{equation*}
		\cF_\ell:=\cap_{\iota\le \ell}\{\tau^{e_{\iota},e_{j_\iota}}\le t \}\cap \cP^{\rm tot}(\iota),\qquad \ell=\{2,\dots, C\}.
	\end{equation*}
	By construction, regardless of the specific ordering of the elements in the set $B_I$, there is a collection of at least $C/2$ indices $\ell\in\{1,\dots, C \}$ such that
	\begin{equation*}
		|\Xi_\ell|:=\big|\{e_{\ell},e_{j_\ell} \}\cap \{e_{1},e_{j_1},\dots,e_{{\ell-1}},e_{j_{\ell-1}} \}  \big|\le 1.
	\end{equation*}
	We claim that, for every $\ell\in\{1,\dots,C\}$ for which $|\Xi_\ell|=0$, 
	\begin{equation}
		\label{eee1}
		\Prob_{\otimes_C \nu}\left(\tau^{e_{\ell},e_{j_\ell}}\le t\mid \cF_{\ell-1}, |\Xi_\ell|=0 \right)
		\le \Prob_{\nu \otimes \nu}\left(\tau^{e,f}\le t\right) + \frac{4C}{n} = O\left(\frac{C+t}{n}\right).
	\end{equation}
	Indeed, 
	\begin{equation*}
		\begin{aligned}
			&\Prob_{\otimes_C \nu}\left(\tau^{e_{\ell},e_{j_\ell}}\le t\mid \cF_{\ell-1}, |\Xi_\ell|=0 \right)\\
			&\leq \Prob_{\otimes_C \nu}\left(\tau^{e_{\ell},e_{j_\ell}}\le t, \{e_{\ell}^-,e_{\ell}^+,e_{j_\ell}^-,e_{j_\ell}^+\}
			\cap \cup_{\iota<\ell}\{e_{\iota}^-,e_{\iota}^+,e_{j_\iota}^-,e_{j_\iota}^+ \}=\emptyset \mid \cF_{\ell-1}, |\Xi_\ell|=0 \right) 
			+ \frac{4\ell}{n}\\
			&= \Prob_{\otimes_{K_n} \nu}\left(\tau^{e_{i_\ell},e_{j_\ell}}\le t, \{e_{\ell}^-,e_{\ell}^+,e_{j_\ell}^-,e_{j_\ell}^+\}
			\cap \cup_{\iota<\ell}\{e_{\iota}^-,e_{\iota}^+,e_{j_\iota}^-,e_{j_\iota}^+ \}=\emptyset  \right) + \frac{4\ell}{n}\\
			&\leq \Prob_{\otimes_
				{K_n} \nu}\left(\tau^{e_{i_\ell},e_{j_\ell}}\le t\right)+\frac{4C}{n}.
		\end{aligned}
	\end{equation*}
	Therefore, by \eqref{ccc}, we conclude that, under $|\Xi_\ell|=0$,
	\begin{equation}
		\label{eee}
		\Prob_{\otimes_C \nu}\left(\tau^{e_{\ell},e_{j_\ell}}\le t\mid \cF_{\ell-1}, |\Xi_\ell|=0  \right)
		\le 4\Prob_{\pi\otimes\pi}(\tau_{\rm meet}^{x,y}\le t)+\frac{4C}{n},
	\end{equation}
	and \eqref{eee1} follows by inserting \eqref{eq:ccc2} into \eqref{eee}.
	
	\medskip\noindent
	{\bf 3.}
	Assume now that for some $\ell\in\{1,\dots, C\}$ we have $|\Xi_\ell|=1$. More precisely, under $\cF_{\ell-1}\cap \left\{|\Xi_\ell|=1\right\}$, one among $e_{\ell}$ and $e_{j_\ell}$ is still unsampled at step $\ell$. Without loss of generality, assume that $e_{\ell}$ is such an unsampled random edge. By the same argument used to derive \eqref{eee}, and putting
	\begin{equation*}
		\cJ_\ell:=\left\{\{e_{\ell}^-,e_{\ell}^+ \}\bigcap \cup_{\iota< \ell}\{e_{\iota}^-,e_{\iota}^+,e_{j_\iota}^-,e_{j_\iota}^+ \}
		=\emptyset \right\},
	\end{equation*}
	we bound 
	\begin{equation*}
		\begin{split}
			\Prob_{\otimes_C \nu}\left(\tau^{e_{\ell},e_{j_\ell}}\le t,\cP^{\rm tot}(\ell)\mid \cF_{\ell-1}, |\Xi_\ell|=1 \right)
			&\le 2\:\Prob_{\otimes_C \nu}\left(\tau^{e_{\ell},e_{j_\ell}}\le t, \cP^{\rm tot}(\ell),\cJ_\ell \mid \cF_{\ell-1}, |\Xi_\ell|=1 \right)+\frac{8C}{n}.
		\end{split}
	\end{equation*}
	Arguing as in \cite[Proposition 4.5]{Oaop}, we get that, under $|\Xi_\ell|=1$ with $e_\ell$ unsampled at step $\ell$,
	\begin{equation}
		\label{ddd}
		\begin{aligned}	
			&\Prob_{\otimes_C \nu}\left(\tau^{e_{\ell},e_{j_\ell}}\le t,\cP^{\rm tot}(\ell),\cJ_\ell \:\big\rvert\: \cF_{\ell-1}, |\Xi_\ell|=1 \right)\\
			&\le \Prob_{\otimes_C \nu}\left(\tau^{e_{\ell},e_{j_\ell}}\le t \:\big\rvert\: \cF_{\ell-1},\; |\Xi_\ell|=1, \:  \cJ_\ell,\:\cP^{\rm tot}(\ell)\right)\\
			&= \Prob_{\otimes_C\nu}\left(\tau_{\rm meet}^{e_{\ell}^-,e_{j_\ell}^-}\wedge \tau_{\rm meet}^{e_{\ell}^+,e_{j_\ell}^-}
			\wedge \tau_{\rm meet}^{e_{\ell}^-,e_{j_\ell}^+}\wedge \tau_{\rm meet}^{e_{\ell}^+,e_{j_\ell}^+}\le t \:\big\rvert\: 
			\cF_{\ell-1}, \; |\Xi_\ell|=1 ,\: \cJ_\ell,\:\cP^{\rm tot}(\ell)\right)\\
			&\le 4\Prob_{\otimes_C\nu}\left(\tau_{\rm meet}^{e_{\ell}^-,e_{j_\ell}^-}\le t \:\big\rvert\: \cF_{\ell-1},\: |\Xi_\ell|=1 ,\: \cJ_\ell,\:\cP^{\rm tot}(\ell)\right)\\		
			&= 4\Expect\left[\Prob_{\otimes_C\nu}\left(\tau_{\rm meet}^{e_{\ell}^-,e_{j_\ell}^-}\le t \:\big\vert\: \cF_{\ell-1},\: |\Xi_\ell|=1 ,\: 
			\cJ_\ell,\:\cP^{\rm tot}(\ell),\:\cP_{e_{\ell}^-},\:\cP_{e_{j_\ell}^-}\right) \bigg \vert  \cF_{\ell-1},\: |\Xi_\ell|=1 ,\: \cJ_\ell, \: \cP^{\rm tot}(\ell) \right],
		\end{aligned}
	\end{equation}
	where the second inequality follows from the equality of the marginals. Note that, conditionally on $\cF_{\ell-1}$ and $\cJ_\ell$, for all $s\le t$ the position of the random walk starting at $e_\ell^-$ is distributed according to $\pi$. On the other hand, the position of the random walk starting at $e_{j_\ell}^-$ at time $s$ is distributed in a non-trivial way. Moreover, the positions of the two random walks at time $s$ are independent. Hence
	\begin{equation}
		\label{ggg}
		\begin{split}
			&\Prob_{\otimes_C\nu}\left(\tau_{\rm meet}^{e_{\ell}^-,e_{j_\ell}^-}	\le t \:\big\rvert\: \cF_{\ell-1},\: |\Xi_\ell|=1 ,\:  
			\cJ_\ell,\:\cP^{\rm tot}(\ell),\:\cP_{e_{\ell}^-},\:\cP_{e_{j_\ell}^-}\right)\\
			&\le \sum_{s\in \cP_{e^-_{\ell}}\cup\cP_{e^-_{j_\ell}}}\Prob_{\otimes_C\nu}\left(X^{(e_{\ell}^-)}_s=X_s^{(e_{j_\ell}^-)} 
			\:\big\rvert\: \cF_{\ell-1}, \: |\Xi_\ell|=1 ,\: \cJ_\ell,\:\cP^{\rm tot}(\ell),\:\cP_{e_{\ell}^-},\:\cP_{e_{j_\ell}^-}\right)\\
			&\le \sum_{s\in \cP_{e^-_{\ell}}\cup\cP_{e^-_{j_\ell}}}\max_{\lambda}\sum_{x\in[n]}\lambda (x)\pi(x)
			\le \frac{|\cP_{e^-_{\ell}}\cup\cP_{e^-_{j_\ell}}|}{n}.
		\end{split}
	\end{equation}
	Inserting \eqref{ggg} into the expectation in \eqref{ddd}, we obtain
	\begin{equation}
		\label{hhh}
		\begin{split}
			&\Prob_{\otimes_C \nu}\left(\tau^{e_{i_\ell},e_{j_\ell}}\le t,\cP^{\rm tot}(\ell),\cJ_\ell \:\big\rvert\: \cF_{\ell-1},\: |\Xi_\ell|=1 \right)
			\le  \frac{8 n^{\varepsilon/2} t}{n}.
		\end{split}
	\end{equation}
	Inserting \eqref{eee}--\eqref{hhh} into \eqref{fff}, we obtain
	\begin{equation}\label{eq:GammaC}
		\Prob_{\otimes_C \nu}\left(\cE_t(I,B_I)	\cap\cP^{\rm tot}\right)\le \left(\frac{8(n^{\varepsilon/2} t + C)}{n} \right)^{C/2}
		\le \left(\sqrt{\frac{8(n^{\varepsilon/2} t + K)}{n}} \right)^{C} =: \Gamma^{C}.
	\end{equation}
	
	\medskip\noindent
	{\bf 4.}
	By \eqref{eq:first-bound}, \eqref{eq:GammaC} and the fact that there are $(C-1)^{C}$ different ways to select the set $B_I$, we are now ready to conclude that $\whp$
	\begin{equation*}
		\begin{aligned}
			\Prob_{\otimes_{K_n} \nu}\left( \cE_{t,\gamma}\cap\cP^{\rm tot} \right)
			&\le \sum_{C=\gamma K_n}^{K_n}\binom{K_n}{C} C^C \Gamma^{C}
			\le \sum_{C=\gamma K_n}^{K_n} (K_n^2\Gamma)^C.
		\end{aligned}
	\end{equation*}
	Since, by our choice of $t$ and $K_n$,
	\begin{equation*}
		K_n^2\Gamma\le n^{-\varepsilon/5}
	\end{equation*}
	for all $n$ large enough, we obtain
	\begin{equation*}
		\begin{aligned}
			\Prob_{\otimes_{K_n} \nu}\left( \cE_{t,\gamma}\cap\cP^{\rm tot} \right)&\le K_n n^{-\frac{\varepsilon\gamma}{5} K_n},
		\end{aligned}
	\end{equation*}
	concluding the proof.
\end{proof}


\appendix

\section{Auxilliary facts for c\`adl\`ag processes}
\label{app:aux}

In the next lemma, for an arbitrary $T>0$, we let $[0,T] \times \mathcal{D}[0,T]$ denote the product of $[0,T]$ with the c\`adl\`ag space $\mathcal{D}[0,T]$ endowed with the $J_1$-Skorohod topology.

\begin{lemma}
	\label{lemma:continuity}
	Consider the function $h\colon\, [0,T] \times \mathcal{D}[0,T] \to \R$ defined via $h(s, \phi) = \phi(s)$. Every point in $[0,T] \times \mathcal{C}[0,T]$ is a continuity point of $h$.
\end{lemma}

\begin{proof}
	Assume that $(s_{n}, \phi_{n})$ is a sequence that converges to $(s, \phi)$, with $\phi \in \mathcal{C}[0,T]$. We need to verify that $h(s_{n}, \phi_{n})$ converges to $h(s, \phi)$. Indeed, if $\phi_{n}$ converges to $\phi$, then there exists a sequence of increasing functions $\lambda_{n}: [0,T] \to [0,T]$ with $\lambda_{n}(0) = 0$ and $\lambda_{n}(T)=T$ such that $\lambda_{n}(u) \to u$ and $\phi_{n}(u) -\phi(\lambda_{n}(u)) \to 0$ uniformly over $u \in [0,T]$. From this we obtain
	\begin{equation*}
		||\phi_{n}-\phi||_{\infty} \leq ||\phi_{n}-\phi \circ \lambda_{n}||_{\infty}+||\phi\circ \lambda_{n}-\phi||_{\infty} \to 0,
	\end{equation*}
	since any function in $\mathcal{C}[0,T]$ is uniformly continuous. In particular, $\phi_{n}$ converges to $\phi$ uniformly, which readily implies that $\phi_{n}(s_{n}) \to \phi(s)$ and concludes the proof of the lemma.
\end{proof}

\begin{proposition}
	\label{prop:convergence}
	Assume that $(X^{n}_{u})_{u \in [0,T]}$ is a sequence of c\`adl\`ag processes that converge in distribution to a process $(X_{u})_{u \in [0,T]}$ that has almost surely continuous trajectories. For any sequence $t_{n} \to t \in [0,T]$, $X^{n}_{t_{n}}$ converges in distribution to $X_{t}$.
\end{proposition}

\begin{proof}
	From the assumptions of the proposition, it follows that $((X^{n}_{u})_{u \in [0,T]}, t_{n})_{n \in \N}$ converges in distribution to $((X_{u})_{u \in [0,T]}, t)$. Let $h\colon\, [0,T] \times \mathcal{D}[0,T] \to \R$ as in Lemma~\ref{lemma:continuity}. Note that, since $(X_{u})_{u \in [0,T]}$ is almost surely continuous, it follows that $X^{n}_{t_{n}} = h(t_{n}, (X^{n}_{u})_{u \in [0,T]})$ converges in distribution to $X_{t} = h(t, (X_{u})_{u \in [0,T]})$.
\end{proof}


\bibliographystyle{apalike}
\bibliography{bibFVTL.bib}


\end{document}